\documentclass[times,sort&compress,3p]{elsarticle}
\journal{Journal of Multivariate Analysis}

\usepackage{amsmath,amsfonts,amssymb,amsthm,graphicx,hyperref,url}
\usepackage[labelfont=bf]{caption}

\usepackage{mathtools}
\mathtoolsset{showonlyrefs=true} 
\usepackage{dsfont} 
\usepackage{mathrsfs} 
\usepackage{subcaption}

\theoremstyle{plain}
\newtheorem{theorem}{Theorem}

\newtheorem{proposition}{Proposition}
\newtheorem{lemma}{Lemma}
\newtheorem{corollary}{Corollary}

\theoremstyle{definition}

\newtheorem{remark}{Remark}

\newcommand{\N}{\mathbb{N}}
\newcommand{\R}{\mathbb{R}}
\newcommand{\PP}{\mathbb{P}}
\newcommand{\EE}{\mathbb{E}}
\newcommand{\BB}{\mathbb{B}\mathrm{ias}}
\newcommand{\VV}{\mathbb{V}\mathrm{ar}}
\newcommand{\bb}[1]{\boldsymbol{#1}}
\newcommand{\OO}{\mathcal{O}}
\newcommand{\oo}{\mathrm{o}}
\newcommand{\leqdef}{\vcentcolon=}

\newcommand{\rd}{{\rm d}}
\newcommand{\ind}{\mathds{1}}
\newcommand{\e}{\varepsilon}

\begin{document}

\begin{frontmatter}

\title{Asymptotic properties of Bernstein estimators on the simplex}%

\author[a1]{Fr\'ed\'eric Ouimet\texorpdfstring{\corref{cor1}\fnref{fn1}}{)}}%

\address[a1]{California Institute of Technology, Pasadena, USA.}%

\cortext[cor1]{Corresponding author. Email address: \url{ouimetfr@caltech.edu}}

\begin{abstract}
    Bernstein estimators are well-known to avoid the boundary bias problem of traditional kernel estimators.
    The theoretical properties of these estimators have been studied extensively on compact intervals and hypercubes, but never on the simplex, except for the mean squared error of the density estimator in \cite{MR1293514} when $d = 2$.
    The simplex is an important case as it is the natural domain of compositional data.
    In this paper, we make an effort to prove several asymptotic results (bias, variance, mean squared error (MSE), mean integrated squared error (MISE), asymptotic normality, uniform strong consistency) for Bernstein estimators of cumulative distribution functions  and density functions on the $d$-dimensional simplex.
    Our results generalize the ones in \cite{MR2960952} and \cite{MR1910059}, who treated the case $d = 1$, and significantly extend those found in \cite{MR1293514}.
    In particular, our rates of convergence for the MSE and MISE are optimal.
\end{abstract}


\end{frontmatter}

\section{The models}\label{sec:models}

The $d$-dimensional (unit) simplex and its interior are defined by
\begin{equation}\label{eq:def.simplex}
    \mathcal{S} \leqdef \big\{\bb{x}\in [0,1]^d: \|\bb{x}\|_1 \leq 1\big\}, \qquad \mathrm{Int}(\mathcal{S}) \leqdef \big\{\bb{x}\in (0,1)^d: \|\bb{x}\|_1 < 1\big\},
\end{equation}
where $\|\bb{x}\|_1 \leqdef \sum_{i=1}^d |x_i|$.
For any (joint) cumulative distribution function $F$ on $\mathcal{S}$ (meaning that it takes the values $0$ or $1$ outside $\mathcal{S}$), define the Bernstein polynomial of order $m$ for $F$ by
\begin{equation}\label{eq:Bernstein.polynomial}
    F_m^{\star}(\bb{x}) \leqdef \sum_{\bb{k}\in \N_0^d \cap m \mathcal{S}} F(\bb{k}/m) P_{\bb{k},m}(\bb{x}), \quad \bb{x}\in \mathcal{S}, ~m \in \N,
\end{equation}
where the weights are the following probabilities from the $\mathrm{Multinomial}\hspace{0.2mm}(m,\bb{x})$ distribution:
\begin{equation}\label{eq:multinomial.probability}
    P_{\bb{k},m}(\bb{x}) \leqdef \frac{m!}{(m - \|\bb{k}\|_1)! \prod_{i=1}^d k_i!} \cdot (1 - \|\bb{x}\|_1)^{m - \|k\|_1} \prod_{i=1}^d x_i^{k_i}, \quad \bb{k}\in \N_0^d \cap m\mathcal{S}.
\end{equation}
The Bernstein estimator of $F$, denoted by $F_{n,m}^{\star}$, is the Bernstein polynomial of order $m$ for the empirical cumulative distribution function
\begin{equation}
    F_n(\bb{x}) \leqdef n^{-1} \sum_{i=1}^n \ind_{(-\bb{\infty},\bb{x}]}(\bb{X}_i),
\end{equation}
where the observations $\bb{X}_1, \dots, \bb{X}_n$ are assumed to be independent and $F$ distributed.
Precisely, let
\begin{equation}\label{eq:cdf.Bernstein.estimator}
    F_{n,m}^{\star}(\bb{x}) \leqdef \sum_{\bb{k}\in \N_0^d \cap m \mathcal{S}} F_n(\bb{k}/m) P_{\bb{k},m}(\bb{x}), \quad \bb{x}\in \mathcal{S}, ~m,n \in \N.
\end{equation}
It should be noted that the c.d.f.\ estimator in \eqref{eq:cdf.Bernstein.estimator} only makes sense here if the observations' support is contained in a hyperrectangle inside the unit simplex.
If the observations have full support on the unit simplex, then the relevant part of the c.d.f.\ estimator would be on the unit hypercube, in which case results analogous to those in Section~\ref{sec:results.cdf.estimator} can be found in \cite{MR2270097,MR3474765}, the difference being a product of binomial weights replacing the multinomial weight function in \eqref{eq:cdf.Bernstein.estimator}.

For a density $f$ supported on $\mathcal{S}$, we define the Bernstein density estimator of $f$ by
\begin{equation}\label{eq:histogram.estimator}
    \hat{f}_{n,m}(\bb{x}) \leqdef \hspace{-3mm}\sum_{\bb{k}\in \N_0^d \cap (m-1) \mathcal{S}} \frac{(m-1+d)!}{(m-1)!} \left\{\frac{1}{n} \sum_{i=1}^n \ind_{\left(\frac{\bb{k}}{m}, \frac{\bb{k} + 1}{m}\right]}(\bb{X}_i)\right\} P_{\bb{k},m-1}(\bb{x}), \quad \bb{x}\in \mathcal{S}, ~m,n \in \N,
\end{equation}
where $(m-1+d)! / (m-1)!$ is just a scaling factor proportional to the inverse of the volume of the hypercube $(\bb{k} / m, (\bb{k} + 1) / m] \leqdef (k_1 / m, (k_1 + 1) / m] \times \dots \times (k_d / m, (k_d + 1) / m]$.
The reader can easily verify that $\hat{f}_{n,m}$ is a proper density function using the identity
\begin{equation}\label{eq:dirichlet.identity}
    \int_{\mathcal{S}} (1 - \|\bb{x}\|_1)^b \prod_{i=1}^d x_i^{a_i} \rd \bb{x} = \frac{b! \prod_{i=1}^d a_i!}{(b + \sum_{i=1}^d a_i + d)!}, \quad a_i,b\in \N_0.
\end{equation}
If we replace the factor $(m-1+d)! / (m-1)!$ by $m^d$ for mathematical convenience in \eqref{eq:histogram.estimator}, then $\hat{f}_{n,m}$ would still be asymptotically a density. The asymptotic results proved in this paper are almost the same under both definitions of the density estimator: the factor $d (d - 1) f(\bb{x}) / 2$ in \eqref{eq:def:b.x} disappears if we replace $(m-1+d)! / (m-1)!$ by $m^d$ in \eqref{eq:histogram.estimator}.

\begin{remark}
    An alternative way to write the Bernstein density estimator \eqref{eq:histogram.estimator} is as a specific finite mixture of Dirichlet densities:
    \begin{equation}
        \hat{f}_{n,m}(\bb{x}) = \hspace{-3mm}\sum_{\bb{k}\in \N_0^d \cap (m-1) \mathcal{S}} \left\{\frac{1}{n} \sum_{i=1}^n \ind_{\left(\frac{\bb{k}}{m}, \frac{\bb{k} + 1}{m}\right]}(\bb{X}_i)\right\} D(\bb{k}+1,m-\|\bb{k}\|_1)(\bb{x}),
    \end{equation}
    where the value of the density of the $\mathrm{Dirichlet}(\bb{\alpha},\beta)$ distribution at $\bb{x}\in \mathcal{S}$ is
    \begin{equation}
        D(\bb{\alpha},\beta)(\bb{x}) \leqdef \frac{(\beta + \|\bb{\alpha}\|_1 - 1)!}{(\beta - 1)! \prod_{i=1}^d (\alpha_i - 1)!} \cdot (1 - \|\bb{x}\|_1)^{\beta - 1} \prod_{i=1}^d x_i^{\alpha_i - 1}, \quad \alpha_i, \beta > 0.
    \end{equation}
    (The Dirichlet density is illustrated in Fig.~\ref{fig:dirichlet.densities}.)
    This means that the density estimator $\hat{f}_{n,m}$ is part of the larger class of finite Dirichlet mixtures, where more liberty could be placed on the choice of the weights (i.e., the braces for each summand), in the same way that finite beta mixtures, studied for example in \cite{MR2227397,MR2589319,MR3120829,MR3139345,MR3488598,MR3740722,MR3916632}, are generalizations of one-dimensional Bernstein density estimators.
    We leave this point for future research.
    For further information on finite mixture models, we refer the reader to \citet{MR1789474}.
\end{remark}

    \begin{figure}[ht]
        \captionsetup[subfigure]{labelformat=empty}
        \captionsetup{width=0.8\linewidth}
        \centering
        \begin{subfigure}[b]{0.34\textwidth}
            \centering
            \includegraphics[width=\textwidth]{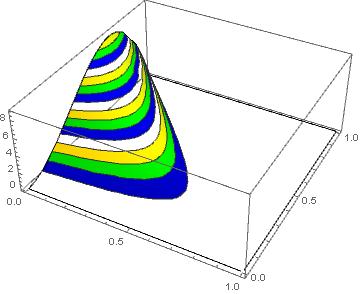}
            \caption{$\alpha_1 = 1$\vspace{3mm}}
            \label{fig:dirichlet.1}
        \end{subfigure}
        \quad
        \begin{subfigure}[b]{0.34\textwidth}
            \centering
            \includegraphics[width=\textwidth]{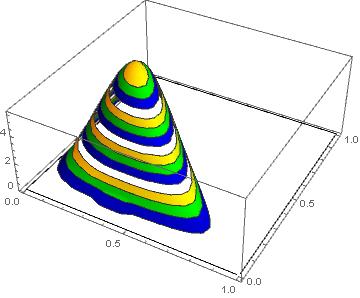}
            \caption{$\alpha_1 = 2$\vspace{3mm}}
            \label{fig:dirichlet.2}
        \end{subfigure}
        \vspace{3mm}
        \begin{subfigure}[b]{0.34\textwidth}
            \centering
            \includegraphics[width=\textwidth]{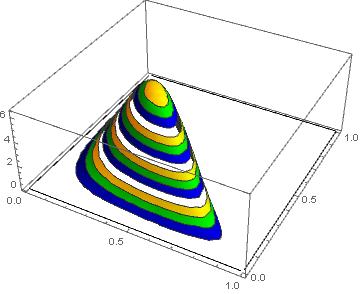}
            \caption{$\alpha_1 = 3$}
            \label{fig:dirichlet.3}
        \end{subfigure}
        \quad
        \begin{subfigure}[b]{0.34\textwidth}
            \centering
            \includegraphics[width=\textwidth]{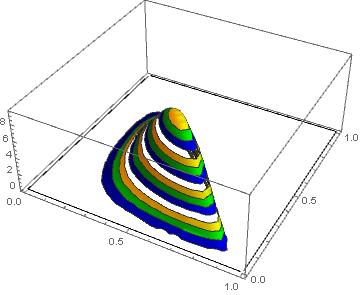}
            \caption{$\alpha_1 = 5$}
            \label{fig:dirichlet.4}
        \end{subfigure}
        \caption{The $\mathrm{Dirichlet}\hspace{0.3mm}(\bb{\alpha} = (\alpha_1,3), \beta = 2)$ density on the two-dimensional simplex, for $\alpha_1 = 1,2,3,5$.}
        \label{fig:dirichlet.densities}
    \end{figure}

\section{Overview of the literature}\label{sec:overview.literature}

    Below, we give a systematic overview of the main line of articles on Bernstein estimation, and then we briefly mention several other branches in which it appeared.
    There might be more details than the reader expects, but this is because the subject is vast and relatively important references are often disjointed or missing in the literature, which makes it hard for newcomers to get a complete chronological account of the progress in the field.

    \citet{MR0397977} was the first to consider Bernstein density estimation on the compact interval $[0,1]$, namely \eqref{eq:histogram.estimator} with $d = 1$.
    In his paper, he computes the asymptotics of the bias, variance and mean squared error (MSE) at each point where the second derivative of $f$ exists (also assuming that $f$ is bounded everywhere). His proof rests on careful Taylor expansions for the density points inside the bulk of the binomial distribution (after rescaling to $[0,1]$) while concentration bounds are applied to show that the contributions coming from outside the bulk are negligible.
    The optimal rate for the MSE is achieved when $m \asymp n^{2/5}$ and shown to be $\OO_x(n^{-4/5})$ for $x\in (0,1)$, and $\OO_x(n^{-3/5})$ at the boundary (where $m \asymp n^{2/5}$ is suboptimal).
    \citet{MR0574548} then considered a density estimator (also called smoothed histogram) on the interval $[0,\infty)$ with Poisson weights instead of binomial weights.
    Assuming that $f$ is in $C^2[0,1]$, the asymptotics of the MSE are derived pointwise and uniformly on compacts, and the optimal rate is again shown to be $\OO_x(n^{-4/5})$ when $m \asymp n^{2/5}$.
    The proof follows the same line as in \cite{MR0397977}, with a slight difference for the concentration bound because the weights are different.
    This latter idea was pushed by \citet{MR0638651}; the results of \citet{MR0397977} were generalized to intervals $I$ of the form $[0,1]$, $[0,\infty)$ and $(-\infty,\infty)$, where the weights of the density estimator are coming from general $n$-fold convolutions (with minor assumptions on the first, second and third moment) instead of binomial weights (Bernoulli convolutions). When the weights are discrete and $f$ is two-times continuously differentiable on $I$, the authors computed the asymptotics of the MSE pointwise and uniformly on compacts using a general local limit theorem for lattice distributions (Edgeworth expansion) instead of the more specific normal approximation to the binomial.
    Following the previous article, \citet{MR0726014} improved tremendously on the precision of the asymptotics of the density estimator by proving a uniform weak law of large numbers, a pointwise central limit theorem (CLT) and a Berry-Esseen type bound.
    The author even shows that the recentered maximum converges in law to a Gumbel distribution.
    For the proof of the uniform weak consistency, the Hadamard product technique allows a separation of the stochastic and non-stochastic part of the density estimator.
    The proof of the CLT is a consequence of the Lindeberg condition for double arrays and the Berry-Esseen bound follows standard arguments for sums of independent random variables.
    In \cite{MR0791719}, various rates (at the level of precision of the law of iterated logarithm) for pointwise and uniform strong consistency were obtained for the density estimator and its derivatives, on intervals of the form $[0,1]$, $[0,\infty)$, $(-\infty,\infty)$.
    The main idea of the paper was to approximate the density estimator and its derivatives by classical kernel estimators, and then to deduce the various strong laws using already established results.

    \citet{MR1293514} was the first to consider Bernstein estimation in the multivariate context.
    Assuming that $f$ is two-times continuously differentiable, he derived asymptotic expressions for the bias, variance and MSE of the density estimators on the two-dimensional unit simplex and the unit square $[0,1]^2$, and also proved their uniform strong consistency and asymptotic normality. (Technically, Tenbusch calls it uniform weak consistency, but his bounds on the deviation probabilities are summable in $n$, so the uniform strong consistency is a trivial consequence of the Borel-Cantelli lemma. The same argument was made in Remark~2 of \cite{MR1985506} for the beta kernel estimator.) He showed that the optimal rate for the MSE is $\OO_{\bb{x}}(n^{-2/3})$ when $d = 2$ and it is achieved when $m \asymp n^{1/3}$ (the optimal rates are also calculated at the boundary).

    \citet{MR1910059} were the first to consider Bernstein estimators for cumulative distribution functions (c.d.f.).
    They complemented the work of Vitale by proving the uniform strong consistency of the c.d.f.\ estimator, and the asymptotic normality and uniform strong consistency of the density estimator on the interval $[0,1]$, under the weaker assumption that $f$ is Lipschitz continuous.
    For the c.d.f.\ estimator, the uniform strong consistency is, a priori, a simple consequence of the Glivenko-Cantelli theorem and the uniform strong consistency of Bernstein polynomials (see, e.g., Theorem~1.1.1 in \cite{MR0864976}), but the authors give explicit rates of convergence, which requires a more careful analysis.
    Their proof follows from a union bound over a partition of small intervals and a concentration bound is applied inside each subinterval. The partition is chosen so that the concentration bounds are summable and the uniform strong consistency follows from the Borel-Cantelli lemma.
    In the case of the density estimator, the proof of the uniform strong consistency is similar to \cite[Theorem~2]{MR1293514}, although we assume that the authors were unaware of this reference at the time of writing. The asymptotic normality follows from a verification of the Lindeberg condition for double arrays.
    Albeit not stated explicitly as theorems, asymptotic expressions for the bias and variance of the density estimator can also be found in the proofs.

    \citet{MR2068610} studied three modified Bernstein density estimators where the empirical density (or empirical histogram function) is replaced by boundary kernel density estimators. Asymptotic expressions are derived for the bias, variance and MISE.
    The estimators are shown to be superior in terms of MISE than Vitale's density estimator and ``equivalent'' to Chen's boundary Beta kernel estimator \cite{MR1718494}.

    \citet{MR2179543} considered three asymmetric kernel density estimators from \cite{MR1794247} (Gamma) and \cite{MR2053071} (Inverse Gaussian and Reciprocal Inverse Gaussian), and the smoothed histogram with Poisson weights introduced by \citet{MR0574548}.
    They showed the uniform weak consistency on compact subsets of $[0,\infty)$ of the estimators as well as the weak consistency of the $L^1$ norm and the pointwise weak consistency at $0$ when the density is unbounded.
    Under several technical conditions, the pointwise weak consistency of the ratio of the estimator to the target density is also shown.
    An application of the smoothers to income data is also performed.

    \citet{MR2270097} generalized the results in \cite{MR1910059} in two ways.
    They considered observations supported on the $d$-dimensional hypercube (technically, the proofs are written for the case $d = 2$)
    and instead of assuming the data to be independent, they controlled the dependence through the notion of strong mixing.
    Precisely, under $\alpha$-mixing of the observations, the authors showed the uniform strong consistency of the c.d.f.\ estimator, and the asymptotic normality and uniform strong consistency of the density estimator.
    The proofs follow the structure in \cite{MR1910059}, but the biggest innovation is the use of a concentration bound for strictly stationnary $\alpha$-mixing processes, coming from the work in \cite{MR0520961}.
    Albeit not stated explicitly as theorems, asymptotic expressions for the bias and variance of the estimators can also be found in the proofs.

    \citet{MR2351744} considered Bernstein density estimation on $[0,1]$ where the underlying density is continuous on $(0,1)$ but unbounded at $0$.
    The uniform strong consistency of the estimator is shown on every compact subset of $(0,1)$ (assuming an appropriate choice of the bandwidth parameter), and the almost-sure convergence is shown at $0$. To our knowledge, they were the first to apply the least-square-cross-validation (LSCV) method for the selection of the bandwidth parameter in the context of Bernstein estimation. Shortly after, \citet{MR2395599} proved $L^1$ bounds for the smoothed histogram on $[0,\infty)$ (i.e., with Poisson weights instead of binomial weights, see \citet{MR0574548}) following closely the previous work in \cite{MR1985506} for the beta kernel estimator.

    \citet{MR2488150} studied the law of iterated logarithm for the supremum of the recentered c.d.f.\ estimator and a slightly weaker notion called the Chung-Smirnov property.
    His results were more precise and more varied with respect to the regularity assumptions on $f$ and the choice of the parameters $m$ and $n$ then the uniform strong consistency result shown in \cite{MR1910059}.
    We should mention that similar uniform strong consistency results at this level of precision were previously obtained in \cite{MR0791719} for the density estimator and its derivatives.
    Assuming that $f$ is four-times continuously differentiable, \citet{MR2662607} studied the asymptotic properties of a modified density estimator which can be written as a weighed difference of Bernstein density estimators on $[0,1]$ (this is known as additive bias correction/reduction).
    He computed the bias and variance in general, and selected the weights in order to reduce the bias by a factor of $m^{-1}$ compared to the original estimator (the variance only increases by a multiplicative constant).
    With the optimal weights, the bias-corrected estimator is shown to achieve the optimal MISE rate of $\OO(n^{-8/9})$.
    In \cite{MR2960952}, Leblanc studied the asymptotic properties of the c.d.f.\ estimator.
    The results complement those of \citet{MR1910059}.
    Asymptotic expressions for the bias, variance, MSE and MISE were derived (see Remark~\ref{eq:correction.Leblanc.Belalia} for the correction of some statements), and the asymptotically optimal bandwidth parameter $m$ is also computed in terms of the sample size $n$, for both the MSE and MISE.
    Leblanc also found interesting expressions for the local (based on MSE) and global (based on MISE) asymptotic deficiency between the empirical c.d.f.\ and the Bernstein c.d.f.\ estimator.
    Finally, the asymptotic normality follows from a verification of the Lindeberg condition for double arrays.
    The paper \cite{MR2925964} is a follow-up to \cite{MR2960952}.
    Assuming that $f$ is two-times continuously differentiable, Leblanc computed the bias, variance and MSE for the c.d.f.\ and density estimators near the boundary of $[0,1]$ (more specifically, for the points $\lambda / m$ and $1 - \lambda / m$ as $m\to \infty$, where $\lambda \geq 0$ is fixed).
    The limiting expressions were given in terms of Bessel functions of the first kind (of order $0$ and $1$ respectively).
    Instead of the usual normal approximation to the binomial, Poisson approximations were used in the proofs. (The Poisson approximation is well known to perform better near the boundary, see, e.g., \citet{MR56861}.) It was already known from \cite{MR0397977} that Bernstein density estimators have decreased bias and increased variance in the boundary region when assuming the shoulder condition $f'(0) = f'(1) = 0$, but Leblanc showed that the c.d.f.\ estimators have decreased bias and increased variance without assuming the shoulder condition.

    \citet{MR3174309} applied a multiplicative bias correction to Bernstein density estimators on $[0,1]$, which is a method of bias reduction originally developed for traditional kernel estimators in \cite{MR585714}.
    They computed the asymptotics of the bias, variance, MSE and MISE.
    The boundary properties of the MSE were also studied for the density estimators with multiplicative bias correction and additive bias correction (from \cite{MR2662607}).

    \citet{MR3412755} showed that the rate of convergence for the uniform norm of the recentered density estimator on $[0,1]$ can be improved (compared to the one in \cite{MR1910059}) if we restrict to a compact subinterval of $(0,1)$.
    According to the author, this implies that the estimator is not minimax under the supremum norm loss on $[0,1]$ if the density is bounded and twice continuously differentiable.
    It should be noted however that Bernstein density estimators are expected to be minimax under the $L^2$ loss as \cite[Theorem~1]{MR2775207} have shown that this was the case for the very similar beta kernel estimators.

    \citet{MR3474765} generalized the results on the bias, variance, MSE and asymptotic normality proved in \cite{MR2960952} to the c.d.f.\ estimator on the unit square $[0,1]^2$.
    (Some of the proofs rely on an incorrect estimate from \citet{MR2960952}, see Remark~\ref{eq:correction.Leblanc.Belalia}.)
    The asymptotic expressions for the general $d$-dimensional hypercube ($d \geq 1$) are also stated in Remark~2 of the paper, as the proof would be an easy adaptation of the one given for $d = 2$ according to the author.
    Belalia's results complemented those proved in \cite{MR2270097} in the same way that the results in \cite{MR2960952} complemented those in \cite{MR1910059}.
    \citet{MR3630225} introduced a two-stage Bernstein estimator for conditional distribution functions.
    The method consists in smoothing a first-stage Nadaraya–Watson or local linear estimator by constructing its Bernstein polynomial.
    Asymptotics of the bias, variance, MSE and MISE are found, and the asymptotic normality is also proved.
    \citet{MR3983257} recently introduced a conditional density estimator based on Bernstein polynomials.
    From this, new estimators for the conditional c.d.f.\ and conditional mean are derived.
    Asymptotics of the bias and variance are found, and the asymptotic normality is also proved.
    In \cite{doi:10.1080/03610926.2020.1734832}, the authors find the bias and variance of the Bernstein estimator for the joint bivariate distribution function when one variable is subject to right-censoring, and also prove the asymptotic normality and the uniform strong consistency.
    Their estimator can be seen as a smooth version of the estimator in \cite{MR1222607}, which was weighted by the jumps of a Kaplan-Meier estimator.

    Here are the five main branches attached to the main line of articles described above:
    \begin{itemize}\setlength\itemsep{0em}
        \item Bayesian estimation, see, e.g., \cite{MR1703623,MR1712051,MR1873330,MR1881846,MR3390138,MR3671772};
        \item Copulas, see, e.g.,  \cite{MR2061727,MR2364124,Bouezmarni_Rombouts_Taamouti_2009,MR2557614,MR2879763,MR3143795,MR3147339,MR3493523,MR3668546,MR3635017,Scheffer2015phd,MR3406207,MR3818458,doi:10.1080/03610918.2020.1859535,MR4085745};
        \item Nonparametric regression, see, e.g., \cite{MR0858109,MR1437794,Rafajlowicz_Skubalska-Rafajlowicz_1999,MR2176996,MR2459189,MR2782409,MR2853755,MR2915158};
        \item Finite beta mixtures, see, e.g., \cite{MR2227397,MR2589319,MR3120829,MR3139345,MR3488598,MR3740722,MR3916632};
        \item Quantile estimation, see, e.g., \cite{doi:10.1016/0167-9473(87)90061-2,MR1353890,MR1450503,MR1681452,MR1963007}.
    \end{itemize}

    Various other statistical topics related to Bernstein estimators are treated, for example, in
    \cite{MR0755576, 
    MR1381202, 
    MR2153833, 
    MR2235846, 
    MR2440398, 
    MR2750607, 
    MR2769276, 
    Schellhase2012phd, 
    MR3197796, 
    MR3292798, 
    MR3563528, 
    Tencaliec2017phd, 
    MR3740720, 
    arXiv:1911.07087, 
    doi:10.1080/03610918.2019.1699571, 
    MR3899096, 
    MR3899473, 
    MR3964528, 
    Hanebeck2020master, 
    doi:10.1080/03610926.2019.1709872, 
    doi:10.1007/s11009-020-09829-3, 
    Lyu2020master, 
    MR4130895, 
    doi:10.1007/s13163-021-00384-0, 
    MR4198443, 
    doi:10.1007/s10463-020-00783-y}. 

    It should be mentioned that beta kernels, introduced by \citet{doi:10.2307/2347365} (as a particular case of more general Dirichlet kernels on the simplex) and first studied theoretically in \cite{MR1685301,MR1718494,MR1742101,MR1985506}, are sort of continuous analogues of Bernstein density estimators.
    As such, they share many of the same asymptotic properties (with proper reparametrization) and the literature on beta kernels, and the more general class of asymmetric kernels, has parallelled that of Bernstein estimators in the past twenty years.
    For a recap of the vast literature on asymmetric kernels, we refer the reader to \citet{MR3821525}, or \citet{arXiv:2002.06956}, where Dirichlet kernels are studied theoretically for the first time.

\section{Contribution, outline and notation}\label{sec:contribution.and.outline}

    \subsection{Contribution}

    In this paper, our  contribution is to find asymptotic expressions for the bias, variance, MSE and MISE for the Bernstein c.d.f.\ and density estimators on the $d$-dimensional simplex, defined respectively in \eqref{eq:cdf.Bernstein.estimator} and \eqref{eq:histogram.estimator}, and also prove their asymptotic normality and uniform strong consistency.
    We deduce the asymptotically optimal bandwidth parameter $m$ using the expressions for the MSE and MISE as well.
    These theoretical results generalize the ones in \cite{MR2960952} and \cite{MR1910059}, who treated the case $d = 1$, and significantly extend those found in \cite{MR1293514}, who calculated the MSE for the density estimator when $d = 2$.
    In particular, our rates of convergence for the MSE and MISE are optimal, as they coincide (assuming the identification $m^{-1} \approx h^2$) with the rates of convergence for the MSE and MISE of traditional multivariate kernel estimators, studied for example in \cite{MR0740865}.
    In contrast to other methods of boundary bias reduction (such as the reflection boundary technique or boundary kernels (see, e.g., \cite[Chapter 6]{MR3329609}), this property is built-in for Bernstein estimators, which makes them one of the easiest to use in the class of estimators that are asymptotically unbiased near (and on) the boundary. Bernstein estimators are also non-negative everywhere on their domain, which is definitely not the case of many estimators corrected for boundary bias.
    This is another reason for their desirability.
    The boundary properties of the Bernstein c.d.f.\ and density estimators are studied in the companion paper \cite{arXiv:2006.11756}.
    Bandwidth selection methods will be investigated in future work.

    \subsection{Outline}

    In Section~\ref{sec:results.cdf.estimator} and Section~\ref{sec:results.density.estimator}, we state the theoretical results for the c.d.f.\ estimator and the density estimator, respectively.
    The proofs are given in Section~\ref{sec:proofs.results.cdf.estimator} and Section~\ref{sec:proofs.results.density.estimator}.
    Some technical lemmas and tools are gathered in Section~\ref{sec:tools}.

    \subsection{Notation}

    Throughout the paper, the notation $u = \OO(v)$ means that $\limsup |u / v| < C < \infty$ as $m$ or $n$ tends to infinity, depending on the context.
    The positive constant $C$ can depend on the target c.d.f.\ $F$, the target density $f$ or the dimension $d$, but no other variable unless explicitly written as a subscript. The most common occurrence is a local dependence of the asymptotics with a given point $\bb{x}$ on the simplex, in which case we would write $u = \OO_{\bb{x}}(v)$.
    In a similar fashion, the notation $u = \oo(v)$ means that $\lim |u / v| = 0$ as $m$ or $n$ tends to infinity.
    Subscripts indicate which parameters the convergence rate can depend on.
    The symbol $\mathscr{D}$ over an arrow `$\longrightarrow$' will denote the convergence in law (or distribution).
    Finally, the bandwidth parameter $m = m(n)$ is always implicitly a function of the number of observations, the only exceptions being in Proposition~\ref{prop:uniform.strong.consistency}, Proposition~\ref{prop:uniform.strong.consistency.density}, and the related proofs.

\section{Results for the c.d.f.\ estimator \texorpdfstring{$F_{n,m}^{\star}$}{F\_\{n,m\}\_star}}\label{sec:results.cdf.estimator}

Except for Theorem~\ref{thm:Theorem.2.1.Babu.Canty.Chaubey}, we assume the following everywhere in this section:
\begin{align}\label{eq:assump:F}
    &\bullet \quad \hspace{-2mm} \text{The c.d.f.\ $F$ is twice continuously differentiable on $\mathcal{S}$.}
\end{align}

We start by stating a multidimensional version of Weierstrass's theorem (the proof is in Section~\ref{sec:proofs.results.cdf.estimator}) for the uniform convergence of Bernstein polynomials (\cite{Bernstein_1912}), where the asymptotics of the error term is explicit.
For a proof in the unidimensional setting, see, e.g., Section~1.6.1 in \cite{MR0864976}.

\begin{proposition}\label{prop:uniform.strong.consistency}
    Assume that \eqref{eq:assump:F} holds.
    We have, uniformly for $\bb{x}\in \mathcal{S}$,
    \begin{equation}\label{eq:prop:uniform.strong.consistency}
        F_m^{\star}(\bb{x}) = F(\bb{x}) + m^{-1} B(\bb{x}) + \oo(m^{-1}), \quad m\to \infty,
    \end{equation}
    where
    \begin{equation}\label{eq:def:B.x}
        B(\bb{x}) \leqdef \frac{1}{2} \sum_{i,j=1}^d \big(x_i \ind_{\{i = j\}} - x_i x_j\big) \frac{\partial^2}{\partial x_i \partial x_j} F(\bb{x}).
    \end{equation}
\end{proposition}

The asymptotics of the bias and variance for univariate Bernstein c.d.f.\ estimators were first proved in \cite{MR2960952}.
The theorem below extends this to the multidimensional setting.

\begin{theorem}[Bias and variance]\label{thm:bias.var}
    Assume that \eqref{eq:assump:F} holds.
    We have
    \begin{align}
        \BB[F_{n,m}^{\star}(\bb{x})] &= \EE[F_{n,m}^{\star}(\bb{x})] - F(\bb{x}) = m^{-1} B(\bb{x}) + \oo(m^{-1}), \quad \forall \bb{x}\in \mathcal{S}, \label{eq:thm:bias.var.eq.bias} \\[1mm]
        \VV(F_{n,m}^{\star}(\bb{x})) &= n^{-1} \sigma^2(\bb{x}) - n^{-1} m^{-1/2} V(\bb{x}) + \oo_{\bb{x}}(n^{-1} m^{-1/2}), \quad \forall \bb{x}\in \mathrm{Int}(\mathcal{S}), \label{eq:thm:bias.var.eq.var}
    \end{align}
    as $n\to \infty$, where
    \begin{equation}\label{eq:def:sigma.2.V}
        \sigma^2(\bb{x}) \leqdef F(\bb{x}) (1 - F(\bb{x})), \qquad V(\bb{x}) \leqdef \sum_{i=1}^d \frac{\partial}{\partial x_i} F(\bb{x}) \sqrt{\frac{x_i (1 - x_i)}{\pi}}.
    \end{equation}
\end{theorem}

\begin{remark}\label{rem:error.Leblanc}
    In \cite{MR2960952}, the function $V(x)$ should be equal to $f(x) \sqrt{x (1 - x)/\pi}$ instead of $f(x) \sqrt{2x (1 - x)/\pi}$; this is due to the erroneous estimate in the statement of Lemma~2\hspace{0.3mm}(iv) in \cite{MR2960952}.
    This error has spread to at least 15 papers/theses who relied on the estimate; the list appears with suggested corrections in Appendix~B of \cite{MR4213687}.
\end{remark}

\begin{corollary}[Mean squared error]\label{cor:bias.var.implies.MSE}
    Assume that \eqref{eq:assump:F} holds.
    We have, as $n\to \infty$ and for $\bb{x}\in \mathrm{Int}(\mathcal{S})$,
    \begin{equation}
        \begin{aligned}
            \mathrm{MSE}(F_{n,m}^{\star}(\bb{x})) &= n^{-1} \sigma^2(\bb{x}) - n^{-1} m^{-1/2} V(\bb{x}) + m^{-2} B^2(\bb{x}) + \oo_{\bb{x}}(n^{-1} m^{-1/2}) + \oo(m^{-2}).
        \end{aligned}
    \end{equation}
    In particular, if $V(\bb{x}) \cdot B(\bb{x}) \neq 0$, the asymptotically optimal choice of $m$, with respect to $\mathrm{MSE}$, is
    \begin{equation}
        m_{\mathrm{opt}}(\bb{x}) = n^{2/3} \left[\frac{4 B^2(\bb{x})}{V(\bb{x})}\right]^{2/3},
    \end{equation}
    in which case
    \begin{equation}
        \mathrm{MSE}[F_{n,m_{\mathrm{opt}}}^{\star}(\bb{x})] = n^{-1} \sigma^2(\bb{x}) - n^{-4/3} \, \frac{3}{4} \left[\frac{V^4(\bb{x})}{4 B^2(\bb{x})}\right]^{1/3} \hspace{-2mm} + \oo_{\bb{x}}(n^{-4/3}),
    \end{equation}
    as $n\to \infty$.
\end{corollary}

By integrating the MSE and showing that the contributions coming from points near the boundary are negligible, we obtain the following result.

\begin{theorem}[Mean integrated squared error]\label{thm:MISE.optimal}
    Assume that \eqref{eq:assump:F} holds.
    We have, as $n\to \infty$,
    \begin{equation}
        \begin{aligned}
            \mathrm{MISE}[F_{n,m}^{\star}]
            &= n^{-1} \int_{\mathcal{S}} \sigma^2(\bb{x}) \rd \bb{x} - n^{-1} m^{-1/2} \int_{\mathcal{S}} V(\bb{x}) \rd \bb{x} + m^{-2} \int_{\mathcal{S}} B^2(\bb{x}) \rd \bb{x} + \oo(n^{-1} m^{-1/2}) + \oo(m^{-2}).
        \end{aligned}
    \end{equation}
    In particular, if $\int_{\mathcal{S}} B^2(\bb{x}) \rd \bb{x} > 0$, the asymptotically optimal choice of $m$, with respect to $\mathrm{MISE}$, is
    \begin{equation}
        m_{\mathrm{opt}} = n^{2/3} \left[\frac{4 \int_{\mathcal{S}} B^2(\bb{x}) \rd \bb{x}}{\int_{\mathcal{S}} V(\bb{x}) \rd \bb{x}}\right]^{2/3},
    \end{equation}
    in which case, as $n\to \infty$,
    \begin{equation}
        \mathrm{MISE}[F_{n,m_{\mathrm{opt}}}^{\star}] = n^{-1} \int_{\mathcal{S}} \sigma^2(\bb{x}) \rd \bb{x} - n^{-4/3} \, \frac{3}{4} \left[\frac{\big(\int_{\mathcal{S}} V(\bb{x}) \rd \bb{x}\big)^4}{4 \int_{\mathcal{S}} B^2(\bb{x}) \rd \bb{x}}\right]^{1/3} \hspace{-2mm} + \oo(n^{-4/3}).
    \end{equation}
\end{theorem}

A standard verification of the Lindeberg condition for double arrays yields the asymptotic normality.
In the univariate setting, this was first proved by \citet{MR1910059}.

\begin{theorem}[Asymptotic normality]\label{thm:asymptotic.normality}
    Assume that \eqref{eq:assump:F} holds.
    For $\bb{x}\in \mathrm{Int}(\mathcal{S})$ such that $0 < F(\bb{x}) < 1$, we have the following convergence in distribution:
    \begin{equation}\label{thm:asymptotic.normality.OG}
        n^{1/2} (F_{n,m}^{\star}(\bb{x}) - F_m^{\star}(\bb{x})) \stackrel{\mathscr{D}}{\longrightarrow} \mathcal{N}(0,\sigma^2(\bb{x})), \quad m,n\to \infty.
    \end{equation}
    In particular, \eqref{thm:asymptotic.normality.OG} and Proposition~\ref{prop:uniform.strong.consistency} together imply
    \begin{alignat}{3}
        &n^{1/2} (F_{n,m}^{\star}(\bb{x}) - F(\bb{x})) \stackrel{\mathscr{D}}{\longrightarrow} \mathcal{N}(0,\sigma^2(\bb{x})), \quad &&\text{if } n^{1/2} m^{-1}\to 0, \\
        &n^{1/2} (F_{n,m}^{\star}(\bb{x}) - F(\bb{x})) \stackrel{\mathscr{D}}{\longrightarrow} \mathcal{N}(\lambda \, B(\bb{x}),\sigma^2(\bb{x})), \quad &&\text{if } n^{1/2} m^{-1}\to \lambda,
    \end{alignat}
    for any constant $\lambda > 0$.
\end{theorem}

For the next result, we use the notation $\|G\|_{\infty} \leqdef \sup_{\bb{x}\in \mathcal{S}} |G(\bb{x})|$
for any bounded function $G: \mathcal{S}\to \R$, and also
\begin{equation}
    \alpha_n \leqdef (n^{-1} \log n)^{1/2}, \qquad \beta_{n,m} \leqdef \alpha_n \sqrt{\alpha_m}.
\end{equation}
The uniform strong consistency was first proved for univariate Bernstein c.d.f.\ estimators in \cite{MR1910059}.
The idea of the proof was to apply a union bound on small boxes and then prove continuity estimates inside each box using concentration bounds and the assumption on $F$.
The width of the boxes was carefully chosen so that the bounds are summable.
The result then followed by the Borel-Cantelli lemma.
The same strategy is employed here.

\begin{theorem}[Uniform strong consistency]\label{thm:Theorem.2.1.Babu.Canty.Chaubey}
    Let $F$ be continuous on $\mathcal{S}$.
    Then, as $n\to \infty$,
    \begin{equation}\label{eq:thm:Theorem.2.1.Babu.Canty.Chaubey.eq.1}
        \|F_{n,m}^{\star} - F\|_{\infty}\longrightarrow 0, \quad \text{a.s.}
    \end{equation}
    Assume further that $F$ is differentiable on $\mathcal{S}$ and its partial derivatives are Lipschitz continuous.
    Then, for all $m \geq 3$ such that $m^{-1} \leq \beta_{n,m} \leq \alpha_m$ (for example, $2 n^{2/3} / \log n \leq m \leq n^2 / \log n$ works), we have, as $n\to\infty$,
    \begin{equation}\label{eq:thm:Theorem.2.1.Babu.Canty.Chaubey.eq.2}
        \|F_{n,m}^{\star} - F_n\|_{\infty} = \OO(\beta_{n,m}), \quad \text{a.s.}
    \end{equation}
    In particular, for $m = n$, we have $\|F_{n,m}^{\star} - F_n\|_{\infty} = \OO(n^{-3/4} (\log n)^{3/4})$, a.s.
\end{theorem}

\section{Results for the density estimator \texorpdfstring{$\hat{f}_{n,m}$}{hat(f)\_\{n,m\}}}\label{sec:results.density.estimator}

For each result stated in this section, one of the following two assumptions will be used:
\begin{align}
    &\bullet \quad \hspace{-2mm} \text{The density $f$ is Lipschitz continuous on $\mathcal{S}$;} \label{eq:assump:f.density} \\
    &\bullet \quad \hspace{-2mm} \text{The density $f$ is twice continuously differentiable on $\mathcal{S}$.} \label{eq:assump:f.density.2}
\end{align}

We denote the expectation of $\hat{f}_{n,m}(\bb{x})$ by
\begin{equation}
    f_m(\bb{x}) \leqdef \EE[\hat{f}_{n,m}(\bb{x})] = \sum_{\bb{k}\in \N_0^d \cap (m-1) \mathcal{S}} \frac{(m - 1 + d)!}{(m - 1)!} \int_{\left(\frac{\bb{k}}{m}, \frac{\bb{k} + 1}{m}\right]} \hspace{-0.5mm} f(\bb{y}) \rd \bb{y} \, P_{\bb{k},m}(\bb{x}).
\end{equation}
A result analogous to Proposition~\ref{prop:uniform.strong.consistency} for the target density $f$ is the following.

\begin{proposition}\label{prop:uniform.strong.consistency.density}
    Assume that \eqref{eq:assump:f.density.2} holds.
    We have, uniformly for $\bb{x}\in \mathcal{S}$,
    \begin{equation}\label{eq:prop:uniform.strong.consistency.density}
        f_m(\bb{x}) = f(\bb{x}) + m^{-1} b(\bb{x}) + \oo(m^{-1}), \quad m\to \infty,
    \end{equation}
    where
    \begin{equation}\label{eq:def:b.x}
        b(\bb{x}) \leqdef \frac{d (d - 1)}{2} f(\bb{x}) + \sum_{i=1}^d \Big(\frac{1}{2} - x_i\Big) \, \frac{\partial}{\partial x_i} f(\bb{x}) + \frac{1}{2} \sum_{i,j=1}^d \big(x_i \ind_{\{i = j\}} - x_i x_j\big) \frac{\partial^2}{\partial x_i \partial x_j} f(\bb{x}).
    \end{equation}
\end{proposition}

The asymptotics of the bias and variance for univariate Bernstein density estimators were first proved in \cite{MR0397977}.
The case of the two-dimensional simplex was previously treated in \cite{MR1293514}.
The theorem below extends this to all dimensions.

\begin{theorem}[Bias and variance]\label{thm:bias.var.density}
    As $n\to \infty$, we have
    \begin{align}
        \BB[\hat{f}_{n,m}(\bb{x})] = \EE[\hat{f}_{n,m}(\bb{x})] - f(\bb{x}) = m^{-1} b(\bb{x}) + \oo(m^{-1}), \quad \forall \bb{x}\in \mathcal{S}, \label{eq:thm:bias.var.density.eq.bias}
    \end{align}
    only assuming \eqref{eq:assump:f.density.2}, and
    \begin{align}
        \VV(\hat{f}_{n,m}(\bb{x})) = n^{-1} m^{d/2} \psi(\bb{x}) f(\bb{x}) + \oo_{\bb{x}}(n^{-1} m^{d/2}), \quad \forall \bb{x}\in \mathrm{Int}(\mathcal{S}), \label{eq:thm:bias.var.density.eq.var}
    \end{align}
    only assuming \eqref{eq:assump:f.density}, where
    \begin{equation}\label{eq:def.psi}
        \psi(\bb{x}) \leqdef \left[(4\pi)^d (1 - \|\bb{x}\|_1) \prod_{i=1}^d x_i\right]^{-1/2}.
    \end{equation}
\end{theorem}

\begin{corollary}[Mean squared error]\label{cor:bias.var.implies.MSE.density}
    Assume that \eqref{eq:assump:f.density.2} holds, and let $\bb{x}\in \mathrm{Int}(\mathcal{S})$.
    We have, as $n\to \infty$,
    \begin{equation}\label{def:MSE.density}
        \mathrm{MSE}(\hat{f}_{n,m}(\bb{x})) \leqdef \EE\left[\big|\hat{f}_{n,m}(\bb{x}) - f(\bb{x})\big|^2\right] = n^{-1} m^{d/2} \psi(\bb{x}) f(\bb{x}) + m^{-2} b^2(\bb{x}) + \oo_{\bb{x}}(n^{-1} m^{d/2}) + \oo(m^{-2}).
    \end{equation}
    In particular, if $f(\bb{x}) \cdot b(\bb{x}) \neq 0$, the asymptotically optimal choice of $m$, with respect to $\mathrm{MSE}$, is
    \begin{equation}
        m_{\mathrm{opt}}(\bb{x}) = n^{2/(d+4)} \left[\frac{4}{d} \cdot \frac{b^2(\bb{x})}{\psi(\bb{x}) f(\bb{x})}\right]^{2/(d+4)},
    \end{equation}
    with
    \begin{equation}
        \mathrm{MSE}[\hat{f}_{n,m_{\mathrm{opt}}}] = n^{-4 / (d+4)} \left[\frac{\frac{4}{d} + 1}{\big(\frac{4}{d}\big)^{\frac{4}{d+4}}}\right] \frac{\big(\psi(\bb{x}) f(\bb{x})\big)^{4 / (d+4)}}{\big(b^2(\bb{x})\big)^{-d / (d+4)}} + \oo_{\bb{x}}(n^{-4/(d+4)}).
    \end{equation}
    More generally, if $n^{2 / (d+4)} m^{-1} \to \lambda$ for some $\lambda > 0$, then, as $n\to \infty$,
    \begin{equation}
        \mathrm{MSE}[\hat{f}_{n,m}(\bb{x})] = n^{-4 / (d+4)} \big[\lambda^{-d/2} \psi(\bb{x}) f(\bb{x}) + \lambda^2 b^2(\bb{x})\big] + \oo_{\bb{x}}(n^{-4/(d+4)}).
    \end{equation}
\end{corollary}

By integrating the MSE and showing that the contributions coming from points near the boundary are negligible, we obtain the following result.

\begin{theorem}[Mean integrated squared error]\label{thm:MISE.optimal.density}
    Assume that \eqref{eq:assump:f.density.2} holds.
    We have, as $n\to \infty$,
    \begin{equation}\label{def:MISE.density}
        \mathrm{MISE}[\hat{f}_{n,m}]
        \leqdef \int_{\mathcal{S}} \EE\left[\big|\hat{f}_{n,m}(\bb{x}) - f(\bb{x})\big|^2\right] \rd \bb{x}
        = n^{-1} m^{d/2} \int_{\mathcal{S}} \psi(\bb{x}) f(\bb{x}) \rd \bb{x} + m^{-2} \int_{\mathcal{S}} b^2(\bb{x}) \rd \bb{x} + \oo(n^{-1} m^{d/2}) + \oo(m^{-2}).
    \end{equation}
    In particular, if $\int_{\mathcal{S}} b^2(\bb{x}) \rd \bb{x} > 0$, the asymptotically optimal choice of $m$, with respect to $\mathrm{MISE}$, is
    \begin{equation}\label{eq:m.opt.MISE.density}
        m_{\mathrm{opt}} = n^{2/(d+4)} \left[\frac{4}{d} \cdot \frac{\int_{\mathcal{S}} b^2(\bb{x}) \rd \bb{x}}{\int_{\mathcal{S}} \psi(\bb{x}) f(\bb{x}) \rd \bb{x}}\right]^{2/(d+4)},
    \end{equation}
    with
    \begin{equation}
        \mathrm{MISE}[\hat{f}_{n,m_{\mathrm{opt}}}] = n^{-4 / (d+4)} \left[\frac{\frac{4}{d} + 1}{\big(\frac{4}{d}\big)^{\frac{4}{d+4}}}\right] \frac{\big(\int_{\mathcal{S}} \psi(\bb{x}) f(\bb{x}) \rd \bb{x}\big)^{4 / (d+4)}}{\big(\int_{\mathcal{S}} b^2(\bb{x}) \rd \bb{x}\big)^{-d / (d+4)}} + \oo_{\bb{x}}(n^{-4/(d+4)}).
    \end{equation}
    More generally, if $n^{2/(d+4)} m^{-1} \to \lambda$ for some $\lambda > 0$, then, as $n\to \infty$,
    \begin{equation}
        \mathrm{MISE}[\hat{f}_{n,m}] = n^{-4 / (d+4)} \left[\lambda^{-d/2} \int_{\mathcal{S}} \psi(\bb{x}) f(\bb{x}) \rd \bb{x} + \lambda^2 \int_{\mathcal{S}} b^2(\bb{x}) \rd \bb{x}\right] + \oo(n^{-4/(d+4)}).
    \end{equation}
\end{theorem}

By only assuming the Lipschitz continuity of $f$ on $\mathrm{Int}(\mathcal{S})$, we can prove the uniform strong consistency of the density estimator in a manner similar to the proof for the c.d.f.\ estimator.

\begin{theorem}[Uniform strong consistency]\label{thm:Theorem.3.1.Babu.Canty.Chaubey}
    Assume that \eqref{eq:assump:f.density} holds.
    If $2 \leq m \leq n / \log n$ as $n\to \infty$, then
    \begin{equation}\label{eq:thm:Theorem.3.1.Babu.Canty.Chaubey}
        \|f_m - f\|_{\infty} = \OO(m^{-1/2}), \quad\qquad \|\hat{f}_{n,m} - f\|_{\infty} = \OO(m^{d - 1/2} \alpha_n) + \OO(m^{-1/2}), \quad \text{a.s.}
    \end{equation}
    In particular, if $m^{2d - 1} = \oo(n / \log n)$, then $\|\hat{f}_{n,m} - f\|_{\infty}\longrightarrow 0$ a.s.
\end{theorem}

Again, a verification of the Lindeberg condition for double arrays yields the asymptotic normality.

\begin{theorem}[Asymptotic normality]\label{thm:Theorem.3.2.and.3.3.Babu.Canty.Chaubey}
    Assume that \eqref{eq:assump:f.density} holds.
    Let $\bb{x}\in \mathrm{Int}(\mathcal{S})$ be such that $f(\bb{x}) > 0$.
    If $n^{1/2} m^{-d/4}\to \infty$ as $m,n\to \infty$, then
    \begin{equation}\label{eq:thm:Theorem.3.2.and.3.3.Babu.Canty.Chaubey.Prop.1}
        n^{1/2} m^{-d/4} (\hat{f}_{n,m}(\bb{x}) - f_m(\bb{x})) \stackrel{\mathscr{D}}{\longrightarrow} \mathcal{N}(0,\psi(\bb{x}) f(\bb{x})).
    \end{equation}
    If we also have $n^{1/2} m^{-d/4 - 1/2}\to 0$ as $m,n\to \infty$, then Theorem~\ref{thm:Theorem.3.1.Babu.Canty.Chaubey} implies
    \begin{equation}
        n^{1/2} m^{-d/4} (\hat{f}_{n,m}(\bb{x}) - f(\bb{x})) \stackrel{\mathscr{D}}{\longrightarrow} \mathcal{N}(0,\psi(\bb{x}) f(\bb{x})).
    \end{equation}
    Independently of the above rates for $n$ and $m$, if we assume \eqref{eq:assump:f.density.2} instead and $n^{2 / (d+4)} m^{-1} \to \lambda$ for some $\lambda > 0$ as $m,n\to \infty$, then \eqref{eq:thm:Theorem.3.2.and.3.3.Babu.Canty.Chaubey.Prop.1} and Proposition~\ref{prop:uniform.strong.consistency.density} together imply
    \begin{equation}\label{eq:thm:Theorem.3.2.and.3.3.Babu.Canty.Chaubey.Thm.3.3}
        n^{2 / (d+4)} (\hat{f}_{n,m}(\bb{x}) - f(\bb{x})) \stackrel{\mathscr{D}}{\longrightarrow} \mathcal{N}(\lambda \, b(\bb{x}), \lambda^{-d/2} \psi(\bb{x}) f(\bb{x})).
    \end{equation}
\end{theorem}

\begin{remark}
    The rate of convergence for the $d$-dimensional kernel density estimator with i.i.d.\ data and bandwidth $h$ is $\OO_{\bb{x}}(n^{-1/2} h^{-d/2})$ in Theorem~3.1.15 of \cite{MR0740865}, whereas our estimator $\hat{f}_{n,m}$ converges at a rate of $\OO_{\bb{x}}(n^{-1/2} m^{d/4})$.
    Hence, the relation between the scaling factor $m$ of $\hat{f}_{n,m}$ and the bandwidth $h$ of other multivariate kernel estimators is $m \approx h^{-2}$.
\end{remark}

\section{Proof of the results for the c.d.f.\ estimator \texorpdfstring{$F_{n,m}^{\star}$}{F\_\{n,m\}\_star}}\label{sec:proofs.results.cdf.estimator}

\begin{proof}[\bf Proof of Proposition~\ref{prop:uniform.strong.consistency}]
    By the assumption \eqref{eq:assump:F}, a second order mean value theorem yields, for all $\bb{k}\in \N_0^d \cap m\mathcal{S}$ and any given $\bb{x}\in \mathcal{S}$,
    \begin{equation}\label{eq:prop:uniform.strong.consistency.begin}
        F(\bb{k} / m) - F(\bb{x}) = \sum_{i=1}^d \Big(\frac{k_i}{m} - x_i\Big) \frac{\partial}{\partial x_i} F(\bb{x}) + \frac{1}{2} \sum_{i,j=1}^d \Big(\frac{k_i}{m} - x_i\Big) \Big(\frac{k_j}{m} - x_j\Big) \frac{\partial^2}{\partial x_i \partial x_j} F(\bb{\xi}_{\bb{k}}),
    \end{equation}
    where $\bb{\xi}_{\bb{k}}$ is an appropriate vector on the line segment joining $\bb{k} / m$ and $\bb{x}$.
    Using the well-known multinomial identities
    \begin{equation}\label{eq:mean.multinomial.identity}
        \sum_{\bb{k}\in \N_0^d \cap m\mathcal{S}} \Big(\frac{k_i}{m} - x_i\Big) P_{\bb{k},m}(\bb{x}) = 0,
    \end{equation}
    and
    \begin{equation}\label{eq:var.multinomial.identity}
        \sum_{\bb{k}\in \N_0^d \cap m\mathcal{S}} \Big(\frac{k_i}{m} - x_i\Big) \Big(\frac{k_j}{m} - x_j\Big) P_{\bb{k},m}(\bb{x}) = \frac{1}{m} \big(x_i \ind_{\{i = j\}} - x_i x_j\big),
    \end{equation}
    we can multiply \eqref{eq:prop:uniform.strong.consistency.begin} by $P_{\bb{k},m}(\bb{x})$ and sum over all $\bb{k}\in \N_0^d \cap m\mathcal{S}$ to obtain
    \begin{equation}\label{eq:prop:uniform.strong.consistency.Taylor}
        \begin{aligned}
            F_m^{\star}(\bb{x}) - F(\bb{x})
            = \hspace{-1mm} \sum_{\bb{k}\in \N_0^d \cap m\mathcal{S}} (F(\bb{k} / m) - F(\bb{x})) P_{\bb{k},m}(\bb{x})
            &= \frac{1}{2m} \sum_{i,j=1}^d \big(x_i \ind_{\{i = j\}} - x_i x_j\big) \frac{\partial^2}{\partial x_i \partial x_j} F(\bb{x}) \\
            &+ \frac{1}{2} \sum_{i,j=1}^d \sum_{\bb{k}\in \N_0^d \cap m\mathcal{S}} \Big(\frac{k_i}{m} - x_i\Big) \Big(\frac{k_j}{m} - x_j\Big) P_{\bb{k},m}(\bb{x}) \frac{\partial^2}{\partial x_i \partial x_j} (F(\bb{\xi}_{\bb{k}}) - F(\bb{x})).
        \end{aligned}
    \end{equation}
    To conclude the proof, we need to show that the last term is $\oo(m^{-1})$.
    By the uniform continuity of the second order partial derivatives of $F$ on $\mathcal{S}$, we know that
    \begin{equation}
        \max_{1 \leq i,j \leq d} ~ \max_{\bb{x}\in \mathcal{S}} \left|\frac{\partial^2}{\partial x_i \partial x_j} F(\bb{x})\right| \leq M_d, \quad \text{for some constant } M_d > 0,
    \end{equation}
    (where $| \, \cdot \, |$ denotes the absolue value) and we also know that, for all $\e > 0$, there exists $0 < \delta_{\e,d} \leq 1$ such that
    \begin{equation}
        \|\bb{y} - \bb{x}\|_1 \leq \delta_{\e,d} \quad \text{implies} \quad \max_{1 \leq i,j \leq d} \left|\frac{\partial^2}{\partial x_i \partial x_j} F(\bb{y}) - \frac{\partial^2}{\partial x_i \partial x_j} F(\bb{x})\right| \leq \e,
    \end{equation}
    uniformly for $\bb{x},\bb{y}\in \mathcal{S}$.
    By considering the two cases $\|\bb{k} / m - \bb{x}\|_1 \leq \delta_{\e,d}$ and $\|\bb{k} / m - \bb{x}\|_1 > \delta_{\e,d}$, the last term in \eqref{eq:prop:uniform.strong.consistency.Taylor}, in absolute value, is less or equal to
    \begin{equation}\label{eq:prop:uniform.strong.consistency.end}
        \frac{1}{2} \sum_{i,j=1}^d \Bigg\{ \, \e \hspace{-1mm}\sum_{\substack{\bb{k}\in \N_0^d \cap m\mathcal{S} \\ \|\bb{k} / m - \bb{x}\|_1 \leq \delta_{\e,d}}}\hspace{-3mm} \Big|\frac{k_i}{m} - x_i\Big| \Big|\frac{k_j}{m} - x_j\Big| P_{\bb{k},m}(\bb{x}) ~+~ 2M_d \sum_{\ell=1}^d \hspace{-3mm} \sum_{\substack{\bb{k}\in \N_0^d \cap m\mathcal{S} \\ |k_{\ell} / m - x_{\ell}| > \delta_{\e,d} / d}} \hspace{-3mm} P_{\bb{k},m}(\bb{x})\Bigg\}.
    \end{equation}
    By Cauchy-Schwarz and the identity \eqref{eq:var.multinomial.identity}, the first term inside the braces in \eqref{eq:prop:uniform.strong.consistency.end} is
    \begin{equation}\label{eq:prop:uniform.strong.consistency.end.next.1}
        \leq \e \cdot \sqrt{\frac{x_i (1 - x_i)}{m}} \cdot \sqrt{\frac{x_j (1 - x_j)}{m}} \leq \frac{\e}{4m}.
    \end{equation}
    By Bernstein's inequality (see, e.g., Lemma~\ref{lem:Bernstein.inequality}), the second term inside the braces in \eqref{eq:prop:uniform.strong.consistency.end} is
    \begin{equation}\label{eq:prop:uniform.strong.consistency.end.next.2}
        \leq 2M_d \cdot d \cdot 2 \exp\left(-\frac{(m \delta_{\e,d} / d)^2 / 2}{m \cdot 1 + \frac{1}{3} \cdot 1 \cdot (m \delta_{\e,d} / d)}\right) \leq 4 d \, M_d \, e^{-\delta_{\e,d}^2 m / (4 d^2)}.
    \end{equation}
    If we take a sequence $\e = \e(m)$ that goes to $0$ as $m\to \infty$ slowly enough that $1 \geq \delta_{\e(m),d} \geq m^{-1/4}$ (for example), then the bound \eqref{eq:prop:uniform.strong.consistency.end} is $\oo(m^{-1})$.
\end{proof}

\begin{proof}[\bf Proof of Theorem~\ref{thm:bias.var}]
    The expression for the bias of $F_{n,m}^{\star}(\bb{x})$ just follows from Proposition~\ref{prop:uniform.strong.consistency} and the fact that
    \begin{equation}
        \EE\big[F_{n,m}^{\star}(\bb{x})\big] = F_m^{\star}(\bb{x}), \quad \text{for all } \bb{x}\in \mathcal{S}.
    \end{equation}
    To estimate the variance of $F_{n,m}^{\star}(\bb{x})$, note that
    \begin{equation}\label{eq:hat.F.as.mean.Z.i.m}
        F_{n,m}^{\star}(\bb{x}) - F_m^{\star}(\bb{x}) = \sum_{\bb{k}\in \N_0^d \cap m\mathcal{S}} (F_n(\bb{k} / m) - F(\bb{k} / m)) P_{\bb{k},m}(\bb{x}) = \frac{1}{n} \sum_{i=1}^n Z_{i,m},
    \end{equation}
    where
    \begin{equation}\label{eq:def.Z.i.m}
        Z_{i,m} \leqdef \sum_{\bb{k}\in \N_0^d \cap m\mathcal{S}} \big(\ind_{(-\bb{\infty},\frac{\bb{k}}{m}]}(\bb{X}_i) - F(\bb{k} / m)\big) P_{\bb{k},m}(\bb{x}), \quad i\in \{1,\dots,n\}.
    \end{equation}

    For every $m$, the random variables $Z_{1,m}, \dots, Z_{n,m}$ are i.i.d.\ and centered, so that
    \begin{equation}
        \VV(F_{n,m}^{\star}(\bb{x})) = n^{-1} \, \EE[Z_{1,m}^2] = n^{-1} \hspace{0.2mm} \Bigg\{\sum_{\bb{k},\bb{\ell}\in \N_0^d \cap m\mathcal{S}} \hspace{-3mm} F((\bb{k} \wedge \bb{\ell}) / m) P_{\bb{k},m}(\bb{x}) P_{\bb{\ell}\hspace{-0.15mm},m}(\bb{x}) - \big(F_m^{\star}(\bb{x})\big)^2\Bigg\},
    \end{equation}
    where $\bb{k} \wedge \bb{\ell} \leqdef (k_1 \wedge \ell_1, \dots, k_d \wedge \ell_d)^{\top}$.
    Using the expansion in \eqref{eq:prop:uniform.strong.consistency.begin} together with Proposition~\ref{prop:uniform.strong.consistency}, the above is
    \begin{equation}\label{eq:thm:bias.var.Taylor.expansion}
        \begin{aligned}
            n^{-1} \, \cdot ~
            &\Bigg\{F(\bb{x}) (1 - F(\bb{x})) + \OO(m^{-1}) + \sum_{i=1}^d \frac{\partial}{\partial x_i} F(\bb{x}) \hspace{-1mm} \sum_{\bb{k},\bb{\ell}\in \N_0^d \cap m\mathcal{S}} \Big(\frac{k_i \wedge \ell_i}{m} - x_i\Big) P_{\bb{k},m}(\bb{x}) P_{\bb{\ell}\hspace{-0.15mm},m}(\bb{x}) \Bigg. \\[-2mm]
            &\qquad\Bigg. + \sum_{i,j=1}^d \OO\bigg(\sum_{\bb{k},\bb{\ell}\in \N_0^d \cap m\mathcal{S}} \Big|\frac{k_i}{m} - x_i\Big| \Big|\frac{k_j}{m} - x_j\Big| P_{\bb{k},m}(\bb{x}) P_{\bb{\ell}\hspace{-0.15mm},m}(\bb{x})\bigg)\Bigg\}.
        \end{aligned}
    \end{equation}
    The double sum on the first line inside the braces is estimated in \eqref{eq:lem:technical.sums.R.claim.2} of Lemma~\ref{lem:technical.sums.R} and shown to be equal to $-m^{-1/2} \sqrt{x_i (1 - x_i) / \pi} + \oo_{\bb{x}}(m^{-1/2})$, for all $\bb{x}\in \mathrm{Int}(\mathcal{S})$.
    By Cauchy-Schwarz, the identity \eqref{eq:var.multinomial.identity}, and the fact that $\sum_{\bb{\ell}\in \N_0^d \cap m\mathcal{S}} P_{\bb{\ell}\hspace{-0.15mm},m}(\bb{x}) = 1$, the double sum inside the big $\OO$ term is
    \begin{equation}\label{eq:thm:bias.var.end.1}
        \leq \max_{i\in \{1,\dots,n\}} \sum_{\bb{k},\bb{\ell}\in \N_0^d \cap m\mathcal{S}} \Big|\frac{k_i}{m} - x_i\Big|^2 P_{\bb{k},m}(\bb{x}) P_{\bb{\ell}\hspace{-0.15mm},m}(\bb{x}) \leq \frac{1}{m} \max_{i\in \{1,\dots,n\}} x_i (1 - x_i) \leq \frac{1}{4m}.
    \end{equation}
    This ends the proof.
\end{proof}

\begin{proof}[\bf Proof of Theorem~\ref{thm:MISE.optimal}]
    By \eqref{eq:thm:bias.var.Taylor.expansion}, \eqref{eq:thm:bias.var.end.1} and \eqref{eq:thm:bias.var.eq.bias}, we have
    \begin{align}
        \mathrm{MISE}(F_{n,m}^{\star})
        &= \int_{\mathcal{S}} \left(\VV(F_{n,m}^{\star}(\bb{x})) + \BB[F_{n,m}^{\star}(\bb{x})]^2\right) \rd \bb{x} \notag \\
        &= n^{-1} \Bigg\{\int_{\mathcal{S}} F(\bb{x}) (1 - F(\bb{x})) \rd \bb{x} + \OO(m^{-1}) + \sum_{i=1}^d \int_{\mathcal{S}} \frac{\partial}{\partial x_i} F(\bb{x}) \hspace{-1mm} \sum_{\bb{k},\bb{\ell}\in \N_0^d \cap m\mathcal{S}} \Big(\frac{k_i \wedge \ell_i}{m} - x_i\Big) P_{\bb{k},m}(\bb{x}) P_{\bb{\ell}\hspace{-0.15mm},m}(\bb{x}) \rd \bb{x}\Bigg\} \\[-1mm]
        &\quad+ m^{-2} \int_{\mathcal{S}} B^2(\bb{x}) \rd \bb{x} + \oo(m^{-2}).
    \end{align}
    By the assumption \eqref{eq:assump:F}, the partial derivatives $\big(\frac{\partial}{\partial x_i} F\big)_{i=1}^d$ are bounded on $\mathcal{S}$, so Lemma~\ref{lem:technical.sums.R} and the bounded convergence theorem imply
    \vspace{-2mm}
    \begin{align}
        \mathrm{MISE}(F_{n,m}^{\star})
        &= n^{-1} \int_{\mathcal{S}} F(\bb{x}) (1 - F(\bb{x})) \rd \bb{x} - n^{-1} m^{-1/2} \int_{\mathcal{S}} \sum_{i=1}^d \frac{\partial}{\partial x_i} F(\bb{x}) \sqrt{\frac{x_i (1 - x_i)}{\pi}} \, \rd \bb{x} \notag \\
        &\quad+  m^{-2} \int_{\mathcal{S}} B^2(\bb{x}) \rd \bb{x} + \oo(n^{-1} m^{-1/2}) + \oo(m^{-2}).
    \end{align}
    This ends the proof.
\end{proof}

\begin{proof}[\bf Proof of Theorem~\ref{thm:asymptotic.normality}]
    Recall from \eqref{eq:hat.F.as.mean.Z.i.m} that $F_{n,m}^{\star}(\bb{x}) - F_m^{\star}(\bb{x}) = \frac{1}{n} \sum_{i=1}^n Z_{i,m}$ where the $Z_{i,m}$'s are i.i.d.\ and centered random variables.
    Therefore, it suffices to show the following Lindeberg condition for double arrays (see, e.g., Section~1.9.3. in \cite{MR0595165}):
    For every $\e > 0$,
    \begin{equation}\label{eq:thm:asymptotic.normality.Lindeberg.condition}
        s_m^{-2} \, \EE\big[|Z_{1,m}|^2 \ind_{\{|Z_{1,m}| > \e n^{1/2} s_m\}}\big] \longrightarrow 0, \quad n\to \infty.
    \end{equation}
    where $s_m^2 \leqdef \EE\big[|Z_{1,m}|^2\big]$ and where $m = m(n)\to \infty$.
    But this follows from the fact that $|Z_{1,m}| \leq 2$ for all $m$, and $s_m = (n \VV(F_{n,m}^{\star}))^{1/2} \to \sigma(\bb{x})$, $~n\to \infty$, by Theorem~\ref{thm:bias.var}.
\end{proof}

Before proving Theorem~\ref{thm:Theorem.2.1.Babu.Canty.Chaubey}, we need the following lemma.
It is an extension of Lemma~2.2 in \cite{MR2270097}.

\begin{lemma}\label{lem:generalization.Lemma.2.2.Babu.Canty.Chaubey}
    Let $F$ be Lipschitz continuous on $\mathcal{S}$, and let
    \begin{equation}\label{eq:lem:generalization.Lemma.2.2.Babu.Canty.Chaubey.def.N}
        N_{\bb{x},m} \leqdef \left\{\bb{k}\in \N_0^d \cap m\mathcal{S}: \max_{1 \leq i \leq d} \Big|\frac{k_i}{m} - x_i\Big| \leq \alpha_m\right\}.
    \end{equation}
    (You can think of $N_{\bb{x},m}$ as the bulk of the $\mathrm{Multinomial}\hspace{0.2mm}(m,\bb{x})$ distribution; the contributions coming from outside the bulk are small for appropriate $\alpha_m$'s.)
    Then, for all $m \geq 3$ that satisfy $m^{-1} \leq \beta_{n,m} \leq \alpha_m$, we have, as $n\to\infty$,
    \begin{equation}\label{eq:lem:generalization.Lemma.2.2.Babu.Canty.Chaubey.def.H}
        \sup_{\bb{x}\in \mathrm{Int}(\mathcal{S})} \max_{\bb{k}\in N_{\bb{x},m}} \big|F_n(\bb{k} / m) - F(\bb{k} / m) - F_n(\bb{x}) + F(\bb{x})\big| = \OO(\beta_{n,m}), \quad \text{a.s.}
    \end{equation}
\end{lemma}

\begin{proof}[\bf Proof of Lemma~\ref{lem:generalization.Lemma.2.2.Babu.Canty.Chaubey}]
    For all $\bb{k}\in N_{\bb{x},m}$, we have
    \begin{align}\label{eq:lem:generalization.Lemma.2.2.Babu.Canty.Chaubey.begin}
        &\big|F_n(\bb{k} / m) - F(\bb{k} / m) - F_n(\bb{x}) + F(\bb{x})\big| \notag \\[1mm]
        &\leq \sum_{\nu=1}^d \bigg| \, F_n\Big(\frac{k_1}{m},\dots,\frac{k_{\nu-1}}{m},\frac{k_{\nu}}{m},x_{\nu+1},\dots,x_d\Big) - F\Big(\frac{k_1}{m},\dots,\frac{k_{\nu-1}}{m},\frac{k_{\nu}}{m},x_{\nu+1},\dots,x_d\Big) \bigg. \notag \\[-2mm]
        &\hspace{10mm}~~ \bigg. - F_n\Big(\frac{k_1}{m},\dots,\frac{k_{\nu-1}}{m},x_{\nu},x_{\nu+1},\dots,x_d\Big) + F\Big(\frac{k_1}{m},\dots,\frac{k_{\nu-1}}{m},x_{\nu},x_{\nu+1},\dots,x_d\Big) \, \bigg| \notag \\
        &\leq \sum_{\nu=1}^d \, \max_{\substack{i,j\in \N_0 \,: \\ |i - j| \beta_{n,m} \leq 3 \alpha_m}} \bigg| \, F_n\Big(\frac{k_1}{m},\dots,\frac{k_{\nu-1}}{m},j \beta_{n,m},x_{\nu+1},\dots,x_d\Big) - F\Big(\frac{k_1}{m},\dots,\frac{k_{\nu-1}}{m},j \beta_{n,m},x_{\nu+1},\dots,x_d\Big) \bigg. \notag \\[-3mm]
        &\hspace{25mm}~~ \bigg.- F_n\Big(\frac{k_1}{m},\dots,\frac{k_{\nu-1}}{m},i \beta_{n,m},x_{\nu+1},\dots,x_d\Big) + F\Big(\frac{k_1}{m},\dots,\frac{k_{\nu-1}}{m},i \beta_{n,m},x_{\nu+1},\dots,x_d\Big) \, \bigg| + \OO(\beta_{n,m}),
    \end{align}
    where the last inequality comes from our assumption that $F$ is Lipschitz continuous.
    The discretization is illustrated in Fig.~\ref{fig:discretization} below.

    \vspace{-2mm}
    \begin{figure}[ht]
        \centering
        \includegraphics{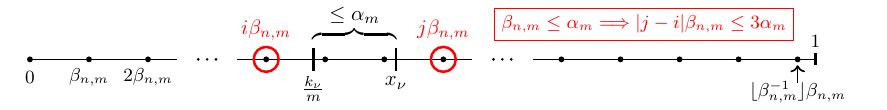}
        \caption{Illustration of the discretization in the proof of Lemma~\ref{lem:generalization.Lemma.2.2.Babu.Canty.Chaubey}. Under the assumption $\beta_{n,m} \leq \alpha_m$, we see that if $|k_{\nu} / m - x_{\nu}| \leq \alpha_m$, then we can select discrete points $i \beta_{n,m}$ and $j \beta_{n,m}$ such that the interval $[i \beta_{n,m}, j \beta_{n,m}]$ has length at most $3 \alpha_m$ and covers the interval $[k_{\nu} / m, x_{\nu}]$.}
        \label{fig:discretization}
    \end{figure}

    \noindent
    Assuming $\ell_{\nu} \beta_{n,m} < y_{\nu} \leq \ell_{\nu}^+ \beta_{n,m}$, for all $\nu = 1,\dots,d$, with the notation $\ell_{\nu}^+ \leqdef \ell_{\nu} + 1$, we have
    \begin{align}
        &\Big| \, F_n(y_1,\dots,y_{\nu-1},j \beta_{n,m},y_{\nu+1},\dots,y_d) - F(y_1,\dots,y_{\nu-1},j \beta_{n,m},y_{\nu+1},\dots,y_d) \Big. \\[-2mm]
        &\Big. ~~ - F_n(y_1,\dots,y_{\nu-1},i \beta_{n,m},y_{\nu+1},\dots,y_d) + F(y_1,\dots,y_{\nu-1},i \beta_{n,m},y_{\nu+1},\dots,y_d) \, \Big| \notag \\
        &\leq ~ \Big| \, F_n(\ell_1^+ \beta_{n,m},\dots,\ell_{\nu-1}^+ \beta_{n,m},j \beta_{n,m},\ell_{\nu+1}^+ \beta_{n,m},\dots,\ell_d^+ \beta_{n,m})
                - F(\ell_1 \beta_{n,m},\dots,\ell_{\nu-1} \beta_{n,m},j \beta_{n,m},\ell_{\nu+1} \beta_{n,m},\dots,\ell_d \beta_{n,m}) \Big. \\[-2mm]
        &\quad\Big. \quad - F_n(\ell_1 \beta_{n,m},\dots,\ell_{\nu-1} \beta_{n,m},i \beta_{n,m},\ell_{\nu+1} \beta_{n,m},\dots,\ell_d \beta_{n,m})
                + F(\ell_1^+ \beta_{n,m},\dots,\ell_{\nu-1}^+ \beta_{n,m},i \beta_{n,m},\ell_{\nu+1}^+ \beta_{n,m},\dots,\ell_d^+ \beta_{n,m}) \, \Big| \notag \\
        &\leq ~ \Big| \, F_n(\ell_1^+ \beta_{n,m},\dots,\ell_{\nu-1}^+ \beta_{n,m},j \beta_{n,m},\ell_{\nu+1}^+ \beta_{n,m},\dots,\ell_d^+ \beta_{n,m})
                - F_n(\ell_1 \beta_{n,m},\dots,\ell_{\nu-1} \beta_{n,m},i \beta_{n,m},\ell_{\nu+1} \beta_{n,m},\dots,\ell_d \beta_{n,m}) \Big. \\[-2mm]
        &\quad\Big. \quad - F(\ell_1^+ \beta_{n,m},\dots,\ell_{\nu-1}^+ \beta_{n,m},j \beta_{n,m},\ell_{\nu+1}^+ \beta_{n,m},\dots,\ell_d^+ \beta_{n,m})
                + F(\ell_1 \beta_{n,m},\dots,\ell_{\nu-1} \beta_{n,m},i \beta_{n,m},\ell_{\nu+1} \beta_{n,m},\dots,\ell_d \beta_{n,m}) \, \Big| \\[-0.5mm]
        &\quad+ \OO(\beta_{n,m}) \notag \\[0.5mm]
        &\leq ~ D_{n,m,\nu} + \OO(\beta_{n,m}),
    \end{align}
    where
    \vspace{-13mm}
    \begin{equation}
        D_{n,m,\nu} \leqdef \max_{\substack{i,j\in \N_0 \,: \\ |i - j| \beta_{n,m} \leq 3 \alpha_m}} \max_{\substack{0 \leq k_p \leq 1 + \lfloor \beta_{n,m}^{-1} \rfloor \\ p\in \{1,\dots,d\}\backslash\{\nu\}}}
            \begin{array}{l}
            \\
            \\
            \\[2mm]
                \Big| \, F_n(k_1 \beta_{n,m},\dots,k_{\nu-1} \beta_{n,m},j \beta_{n,m},k_{\nu+1} \beta_{n,m},\dots,k_d \beta_{n,m}) \\[1mm]
                ~~ - F_n(k_1 \beta_{n,m},\dots,k_{\nu-1} \beta_{n,m},i \beta_{n,m},k_{\nu+1} \beta_{n,m},\dots,k_d \beta_{n,m}) \\[1mm]
                ~~ - F(k_1 \beta_{n,m},\dots,k_{\nu-1} \beta_{n,m},j \beta_{n,m},k_{\nu+1} \beta_{n,m},\dots,k_d \beta_{n,m}) \\
                ~~ + F(k_1 \beta_{n,m},\dots,k_{\nu-1} \beta_{n,m},i \beta_{n,m},k_{\nu+1} \beta_{n,m},\dots,k_d \beta_{n,m}) ~ \Big|.
            \end{array}
    \end{equation}
    Therefore, by \eqref{eq:lem:generalization.Lemma.2.2.Babu.Canty.Chaubey.begin}, it follows that
    \begin{equation}\label{eq:lem:generalization.Lemma.2.2.Babu.Canty.Chaubey.bound.on.H}
        \sup_{\bb{x}\in \mathrm{Int}(\mathcal{S})} \max_{\bb{k}\in N_{\bb{x},m}} \big|F_n(\bb{k} / m) - F(\bb{k} / m) - F_n(\bb{x}) + F(\bb{x})\big| \leq \sum_{\nu=1}^d D_{n,m,\nu} + \OO(\beta_{n,m}).
    \end{equation}
    We want to apply a concentration bound on each $D_{n,m,\nu}$, $\nu = 1,\dots,d$.
    By Bernstein's inequality (see, e.g., Lemma~\ref{lem:Bernstein.inequality}), note that for any $\rho > 0$, any $y_1,\dots,y_{\nu-1},y_{\nu+1},\dots,y_d\in \R$ and any $i,j\in \N_0$ such that $|i - j| \, \beta_{n,m} \leq 3 \alpha_m$, we have, assuming $\beta_{n,m} \leq \alpha_m$,
    \begin{equation}\label{eq:lem:generalization.Lemma.2.2.Babu.Canty.Chaubey.Bernstein.ineq}
        \PP\left(\hspace{-1mm}
                \begin{array}{l}
                    \big| \, F_n(y_1,\dots,y_{\nu-1},j \beta_{n,m},y_{\nu+1},\dots,y_d) \big. \\[0.5mm]
                    ~~ - F_n(y_1,\dots,y_{\nu-1},i \beta_{n,m},y_{\nu+1},\dots,y_d) \\[0.5mm]
                    ~~ - F(y_1,\dots,y_{\nu-1},j \beta_{n,m},y_{\nu+1},\dots,y_d) \\[0.5mm]
                    \big. ~~ + F(y_1,\dots,y_{\nu-1},i \beta_{n,m},y_{\nu+1},\dots,y_d) ~ \big| \geq \rho \beta_{n,m}
                \end{array}
            \hspace{-1mm}\right) \leq 2 \exp\left(-\frac{n^2 \rho^2 \beta_{n,m}^2 / 2}{n \cdot C \cdot 3 \alpha_m + \frac{1}{3} \cdot 1 \cdot n \rho \beta_{n,m}}\right)
            \leq 2 n^{-\rho^2 / (8C)},
    \end{equation}
    where $C\geq \rho$ is a Lipschitz constant for $F$.
    A union bound over $i$, $j$ and the $k_p$'s then yields
    \begin{equation}\label{eq:large.deviation.D.n.m}
        \PP(D_{n,m,\nu} > \rho \beta_{n,m}) \leq \big(2 + \lfloor \beta_{n,m}^{-1} \rfloor\big)^{2 + (d-1)} \cdot 2 n^{-\rho^2 / (8C)}, \quad \nu\in \{1,\dots,d\}.
    \end{equation}
    Since $\beta_{n,m}^{-1} \leq n^2$ (indeed, our assumption $m^{-1} \leq \beta_{n,m}$ implies $\beta_{n,m}^{-1} \leq m$, and the second assumption $\beta_{n,m} \leq \alpha_m$ implies $m \leq n^2$ when $m \geq 3$), we can choose a constant $\rho = \rho(C,d) > 0$ large enough that the right-hand side of \eqref{eq:large.deviation.D.n.m} is summable in $n$, in which case the Borel-Cantelli lemma implies $D_{n,m,\nu} = \OO(\beta_{n,m})$ a.s., as $n\to\infty$.
    The conclusion follows from the bound in \eqref{eq:lem:generalization.Lemma.2.2.Babu.Canty.Chaubey.bound.on.H}.
\end{proof}

\begin{proof}[\bf Proof of Theorem~\ref{thm:Theorem.2.1.Babu.Canty.Chaubey}]
    By the triangle inequality and $\sum_{\bb{k}\in \N_0^d \cap m\mathcal{S}} P_{\bb{k},m}(\bb{x}) = 1$, we have
    \begin{equation}
        \|F_{n,m}^{\star} - F\|_{\infty} \leq \|F_{n,m}^{\star} - F_m^{\star}\|_{\infty} + \|F_m^{\star} - F\|_{\infty} \leq \|F_n - F\|_{\infty} + \|F_m^{\star} - F\|_{\infty}.
    \end{equation}
    The first term on the right-hand side goes to $0$ by the Glivenko-Cantelli theorem, and the second term goes to $0$ by a weak version of Proposition~\ref{prop:uniform.strong.consistency} where $F$ is only assumed to be continuous on $\mathcal{S}$. (To be more precise, after the first equality in \eqref{eq:prop:uniform.strong.consistency.Taylor}, use the uniform continuity of $F$ inside the bulk $N_{\bb{x},m}$ and a concentration bound to show that the contributions coming from outside the bulk are negligible. Alternatively, see Theorem~1.1.1 in \cite{MR0864976}.) This proves \eqref{eq:thm:Theorem.2.1.Babu.Canty.Chaubey.eq.1}.

    For the remainder of the proof, we study the closeness between $F_{n,m}^{\star}$ and the empirical cumulative distribution function $F_n$.
    We assume that $F$ is differentiable on $\mathcal{S}$ and its partial derivatives are Lipschitz continuous.
    By the triangle inequality,
    \begin{equation}\label{eq:thm:Theorem.2.2.Babu.Canty.Chaubey.begin}
        \begin{aligned}
            \|F_{n,m}^{\star} - F_n\|_{\infty}
            &\leq \Bigg\|\sum_{\bb{k}\in N_{\bb{x},m}} (F_n(\bb{k} / m) - F(\bb{k} / m) - F_n(\, \cdot\, ) + F(\, \cdot\, )) P_{\bb{k},m}(\, \cdot\, ) \, \Bigg\|_{\infty} \\
            &+ \Bigg\|\sum_{\bb{k}\in (\N_0^d \cap m\mathcal{S}) \backslash N_{\bb{x},m}} \hspace{-4mm} (F_n(\bb{k} / m) - F(\bb{k} / m) - F_n(\, \cdot\, ) + F(\, \cdot\, )) P_{\bb{k},m}(\, \cdot\, ) \, \Bigg\|_{\infty} + \Bigg\|\sum_{\bb{k}\in \N_0^d \cap m\mathcal{S}} (F(\bb{k} / m) - F(\, \cdot\, )) P_{\bb{k},m}(\, \cdot\, ) \, \Bigg\|_{\infty}.
        \end{aligned}
    \end{equation}
    The first norm is $\OO(\beta_{n,m})$ by Lemma~\ref{lem:generalization.Lemma.2.2.Babu.Canty.Chaubey} (assuming $m \geq 3$ and $m^{-1} \leq \beta_{n,m} \leq \alpha_m$).
    If $B_i\sim \text{Binomial}\hspace{0.2mm}(m,x_i)$, then a union bound, the fact that $\max_{\bb{k}} \|F_n(\bb{k} / m) - F(\, \cdot\, )\|_{\infty} \leq 1$, and Bernstein's inequality (see, e.g., Lemma~\ref{lem:Bernstein.inequality}), yield that the second norm in \eqref{eq:thm:Theorem.2.2.Babu.Canty.Chaubey.begin} is
    \begin{align}\label{eq:thm:Theorem.2.2.Babu.Canty.Chaubey.begin.bound.second.term}
        \leq 2 \cdot \max_{\bb{x}\in \mathcal{S}} \sum_{i=1}^d \PP(|B_i - m x_i| \geq m \alpha_m)
        \leq 2 \cdot \max_{\bb{x}\in \mathcal{S}} \, d \cdot 2 \exp\left(-\frac{m^2 \alpha_m^2 / 2}{m \cdot x_i (1 - x_i) + \frac{1}{3} \cdot 1 \cdot m \alpha_m}\right) \leq 4 \, d \, m^{-1} \leq 4 \, d \, \beta_{n,m}.
    \end{align}
    For the third norm in \eqref{eq:thm:Theorem.2.2.Babu.Canty.Chaubey.begin}, the Lipschitz continuity of the partial derivatives $\big(\tfrac{\partial}{\partial x_i} F\big)_{i=1}^d$ implies that, uniformly for $\bb{x}\in \mathcal{S}$,
    \begin{equation}\label{eq:thm:Theorem.2.2.Babu.Canty.Chaubey.end.third.term}
        F(\bb{k} / m) - F(\bb{x}) = \sum_{i=1}^d \Big(\frac{k_i}{m} - x_i\Big) \frac{\partial}{\partial x_i} F(\bb{x}) + \sum_{i,j=1}^d \OO\bigg(\Big|\frac{k_i}{m} - x_i\Big| \Big|\frac{k_j}{m} - x_j\Big|\bigg).
    \end{equation}
    After multiplying \eqref{eq:thm:Theorem.2.2.Babu.Canty.Chaubey.end.third.term} by $P_{\bb{k},m}(\bb{x})$, summing over all $\bb{k}\in \N_0^d \cap m\mathcal{S}$ and applying the Cauchy-Schwarz inequality, the result is uniformly bounded by $\OO(m^{-1})$ because of the identities \eqref{eq:mean.multinomial.identity} and \eqref{eq:var.multinomial.identity}.
    Since we assumed $m^{-1} \leq \beta_{n,m}$, this proves \eqref{eq:thm:Theorem.2.1.Babu.Canty.Chaubey.eq.2}.
\end{proof}

\section{Proof of the results for the density estimator \texorpdfstring{$\hat{f}_{n,m}$}{hat(f)\_\{n,m\}}}\label{sec:proofs.results.density.estimator}

\begin{proof}[\bf Proof of Proposition~\ref{prop:uniform.strong.consistency.density}]
    Let $\bb{x}\in \mathcal{S}$.
    We follow the proof of Proposition~\ref{prop:uniform.strong.consistency}.
    By using Taylor expansions for any $\bb{k}$ such that $\|\bb{k} / m - \bb{x}\|_1 = \oo(1)$, we obtain
    \begin{align}\label{thm:Theorem.3.2.and.3.3.Babu.Canty.Chaubey.second.asymp.begin.1}
        &m^d \int_{\left(\frac{\bb{k}}{m}, \frac{\bb{k} + 1}{m}\right]} \hspace{-0.5mm} f(\bb{y}) \rd \bb{y} - f(\bb{x}) = f(\bb{k} / m) - f(\bb{x}) + \frac{1}{2m} \sum_{i=1}^d \frac{\partial}{\partial x_i} f(\bb{k} / m) + \OO(m^{-2}) \notag \\[-0.5mm]
        &\hspace{-1mm}= \frac{1}{m} \sum_{i=1}^d (k_i - m x_i) \frac{\partial}{\partial x_i} f(\bb{x}) + \frac{1}{2m} \sum_{i=1}^d \frac{\partial}{\partial x_i} f(\bb{x}) + \oo(m^{-1}) + \frac{1}{2 m^2} \sum_{i,j=1}^d (k_i - m x_i) (k_j - m x_j) \frac{\partial^2}{\partial x_i \partial x_j} f(\bb{x}) (1 + \oo(1)) \notag \\
        &\hspace{-1mm}= \frac{1}{m} \sum_{i=1}^d (k_i - (m - 1) x_i) \frac{\partial}{\partial x_i} f(\bb{x}) + \frac{1}{m} \sum_{i=1}^d \Big(\frac{1}{2} - x_i\Big) \, \frac{\partial}{\partial x_i} f(\bb{x}) + \frac{1}{2} \sum_{i,j=1}^d \Big(\frac{k_i}{m} - x_i\Big) \Big(\frac{k_j}{m} - x_j\Big) \frac{\partial^2}{\partial x_i \partial x_j} f(\bb{x}) (1 + \oo(1)) + \oo(m^{-1}).
    \end{align}
    If we multiply the last expression by $m^{-d} \cdot \frac{(m - 1 + d)!}{(m - 1)!} P_{\bb{k},m-1}(\bb{x})$ and sum over all $\bb{k}\in \N_0^d \cap (m-1)\mathcal{S}$, then the identities \eqref{eq:mean.multinomial.identity} and \eqref{eq:var.multinomial.identity} yield
    \begin{equation}\label{thm:Theorem.3.2.and.3.3.Babu.Canty.Chaubey.second.asymp.end}
        f_m(\bb{x}) - \left(1 + \frac{d (d - 1)}{2m}\right) f(\bb{x}) = 0 + \frac{1}{m} \sum_{i=1}^d \Big(\frac{1}{2} - x_i\Big) \, \frac{\partial}{\partial x_i} f(\bb{x}) + \frac{1}{2m} \sum_{i,j=1}^d \big(x_i \ind_{\{i = j\}} - x_i x_j\big) \frac{\partial^2}{\partial x_i \partial x_j} f(\bb{x}) + \oo(m^{-1}),
    \end{equation}
    assuming that $\|\bb{k} / m - \bb{x}\|_1 = \oo(1)$ decays slowly enough to $0$ that the contributions coming from outside the bulk are negligible (exactly as we did in \eqref{eq:prop:uniform.strong.consistency.end.next.2}).
    This ends the proof.
\end{proof}

\begin{proof}[\bf Proof of Theorem~\ref{thm:bias.var.density}]
    The expression for the bias follows from Proposition~\ref{prop:uniform.strong.consistency.density} and the fact that $\EE[\hat{f}_{n,m}(\bb{x})] = f_m(\bb{x})$ for all $\bb{x}\in \mathcal{S}$.
    In order to compute the asymptotics of the variance, we only assume that $f$ is Lipschitz continuous on $\mathcal{S}$.
    First, note that
    \begin{equation}\label{eq:equality.T.m.n}
        \hat{f}_{n,m}(\bb{x}) - f_m(\bb{x}) = \frac{(m - 1 + d)!}{(m - 1)!} \cdot \frac{1}{n} \sum_{i=1}^n Y_{i,m},
    \end{equation}
    where
    \begin{equation}\label{eq:def.Y.i.m}
        Y_{i,m} \leqdef \hspace{-2mm}\sum_{\bb{k}\in \N_0^d \cap (m - 1)\mathcal{S}} \left[\ind_{\left(\frac{\bb{k}}{m}, \frac{\bb{k} + 1}{m}\right]}(\bb{X}_i) - \int_{\left(\frac{\bb{k}}{m}, \frac{\bb{k} + 1}{m}\right]} \hspace{-0.5mm} f(\bb{y}) \rd \bb{y}\right] P_{\bb{k},m-1}(\bb{x}), \quad i\in \{1,\dots,n\}.
    \end{equation}
    For every $m$, the random variables $Y_{1,m}, \dots, Y_{n,m}$ are i.i.d.\ and centered, so
    \begin{equation}\label{eq:var.to.second.moment.Y.1.m}
        \VV(\hat{f}_{n,m}(\bb{x})) = n^{-1} \left(\frac{(m - 1 + d)!}{(m - 1)!}\right)^2 \, \EE[Y_{1,m}^2],
    \end{equation}
    and it is easy to see that
    \begin{equation}\label{eq:lem:Lemma.3.2.Babu.Canty.Chaubey.begin.var.Y.1.m}
        \EE[Y_{1,m}^2] = \hspace{-2mm}\sum_{\bb{k}\in \N_0^d \cap (m-1)\mathcal{S}} \int_{\left(\frac{\bb{k}}{m}, \frac{\bb{k} + 1}{m}\right]} \hspace{-0.5mm} f(\bb{y}) \rd \bb{y} \, P_{\bb{k},m-1}^2(\bb{x}) - \left(\frac{(m - 1)!}{(m - 1 + d)!} \, f_m(\bb{x})\right)^2.
    \end{equation}
    The second term on the right-hand side of \eqref{eq:lem:Lemma.3.2.Babu.Canty.Chaubey.begin.var.Y.1.m} is $\OO(m^{-2d})$ since the Lipschitz continuity of $f$ and the identity \eqref{eq:var.multinomial.identity} together imply that, uniformly for $\bb{x}\in \mathcal{S}$,
    \begin{align}\label{eq:thm:Theorem.3.1.Babu.Canty.Chaubey.control.T.m}
        f_m(\bb{x}) - f(\bb{x})
        &= \sum_{i=1}^d \OO\left(\sum_{\bb{k}\in \N_0^d \cap (m - 1)\mathcal{S}} \Big|\frac{k_i}{m} - x_i\Big| P_{\bb{k},m-1}(\bb{x})\right) + \OO(m^{-1}) = \OO(m^{-1/2}).
    \end{align}
    For the first term on the right-hand side of \eqref{eq:lem:Lemma.3.2.Babu.Canty.Chaubey.begin.var.Y.1.m}, the Lipschitz continuity of $f$ implies,
    \begin{align}\label{eq:lem:Lemma.3.2.Babu.Canty.Chaubey.middle.a.l.m}
        m^d \int_{\left(\frac{\bb{k}}{m}, \frac{\bb{k} + 1}{m}\right]} \hspace{-0.5mm} f(\bb{y}) \rd \bb{y}
        = f(\bb{k} / m) + \OO(m^{-1})
        = f(\bb{x}) + \OO(m^{-1}) + \sum_{i=1}^d \OO\bigg(\Big|\frac{k_i}{m} - x_i\Big|\bigg),
    \end{align}
    and by the Cauchy-Schwarz inequality, the identity \eqref{eq:var.multinomial.identity} and \eqref{eq:lem:technical.sums.eq.2} in Lemma~\ref{lem:technical.sums}, we have, for all $i\in \{1,\dots,d\}$,
    \begin{equation}\label{eq:lem:Lemma.3.2.Babu.Canty.Chaubey.middle.a.l.m.next}
        \sum_{\bb{k}\in \N_0^d \cap (m-1)\mathcal{S}} \hspace{-1mm} \Big|\frac{k_i}{m} - x_i\Big| \, P_{\bb{k},m-1}^2(\bb{x})
        \leq \sqrt{\sum_{\bb{k}\in \N_0^d \cap (m-1)\mathcal{S}} \hspace{-1mm} \Big|\frac{k_i}{m} - x_i\Big|^2 P_{\bb{k},m-1}(\bb{x})} \sqrt{\sum_{\bb{k}\in \N_0^d \cap (m-1)\mathcal{S}} \hspace{-1mm} P_{\bb{k},m-1}^3(\bb{x})} = \OO(m^{-1/2-d/2}).
    \end{equation}
    Putting \eqref{eq:thm:Theorem.3.1.Babu.Canty.Chaubey.control.T.m}, \eqref{eq:lem:Lemma.3.2.Babu.Canty.Chaubey.middle.a.l.m} and \eqref{eq:lem:Lemma.3.2.Babu.Canty.Chaubey.middle.a.l.m.next} together in \eqref{eq:lem:Lemma.3.2.Babu.Canty.Chaubey.begin.var.Y.1.m} yields
    \begin{equation}\label{eq:lem:Lemma.3.2.Babu.Canty.Chaubey.end}
        m^{3d/2} \, \EE[Y_{1,m}^2] = (f(\bb{x}) + \OO(m^{-1})) \left[m^{d/2} \hspace{-2mm}\sum_{\bb{k}\in \N_0^d \cap (m-1)\mathcal{S}} \hspace{-2mm}P_{\bb{k},m-1}^2(\bb{x})\right] + \OO(m^{-1/2}).
    \end{equation}
    The result follows from \eqref{eq:var.to.second.moment.Y.1.m}, \eqref{eq:lem:Lemma.3.2.Babu.Canty.Chaubey.end} and \eqref{eq:lem:technical.sums.eq.1} in Lemma~\ref{lem:technical.sums}.
\end{proof}

\begin{proof}[\bf Proof of Theorem~\ref{thm:MISE.optimal.density}]
    In Lemma~\ref{lem:lemma.4.Leblanc.2006.tech.report}, it is shown, using the duplication formula for Euler's gamma function and the Chu–Vandermonde convolution for binomial coefficients, that
    \begin{equation}\label{eq:thm:MISE.optimal.density.begin}
        m^{d/2} \int_{\mathcal{S}} \sum_{\bb{k}\in \N_0^d \cap (m-1)\mathcal{S}} \hspace{-2mm}P_{\bb{k},m-1}^2(\bb{x}) \rd \bb{x} = \int_{\mathcal{S}} \psi(\bb{x}) \rd \bb{x} + \OO(m^{-1}).
    \end{equation}
    Together with the almost-everywhere convergence in \eqref{eq:lem:technical.sums.eq.1} of Lemma~\ref{lem:technical.sums}, and the fact that $f$ is bounded, Scheff\'e's lemma (see, e.g., \cite[p.55]{MR1155402}) implies
    \begin{equation}\label{eq:thm:MISE.optimal.density.begin.next}
        m^{d/2} \int_{\mathcal{S}} \sum_{\bb{k}\in \N_0^d \cap (m-1)\mathcal{S}} \hspace{-2mm}P_{\bb{k},m-1}^2(\bb{x}) f(\bb{x}) \rd \bb{x} = \int_{\mathcal{S}} \psi(\bb{x}) f(\bb{x}) \rd \bb{x} + \oo(1).
    \end{equation}
    Therefore, by \eqref{eq:var.to.second.moment.Y.1.m}, \eqref{eq:lem:Lemma.3.2.Babu.Canty.Chaubey.end}, \eqref{eq:thm:MISE.optimal.density.begin.next} and \eqref{eq:thm:bias.var.density.eq.bias}, we have
    \begin{align}
        \mathrm{MISE}(\hat{f}_{n,m})
        &= \int_{\mathcal{S}} \left(\VV(\hat{f}_{n,m}(\bb{x})) + \BB[\hat{f}_{n,m}(\bb{x})]^2\right) \rd \bb{x} \notag \\[1mm]
        &= n^{-1} m^{d/2} \int_{\mathcal{S}} \psi(\bb{x}) f(\bb{x}) \rd \bb{x} + m^{-2} \int_{\mathcal{S}} b^2(\bb{x}) \rd \bb{x} + \oo(n^{-1} m^{d/2}) + \oo(m^{-2}).
    \end{align}
    This ends the proof.
\end{proof}

\begin{proof}[\bf Proof of Theorem~\ref{thm:Theorem.3.1.Babu.Canty.Chaubey}]
    We have already shown that $\|f_m - f\|_{\infty} = \OO(m^{-1/2})$ in \eqref{eq:thm:Theorem.3.1.Babu.Canty.Chaubey.control.T.m}.
    Next, we want to apply a concentration bound to control $\|\hat{f}_{n,m} - f_m\|_{\infty}$.
    Let
    \begin{equation}
        L_{n,m} \leqdef \max_{\bb{k}\in \N_0^d \cap (m-1)\mathcal{S}} \frac{1}{n} \sum_{i=1}^n \left(\ind_{\left(\frac{\bb{k}}{m}, \frac{\bb{k} + 1}{m}\right]}(\bb{X}_i) - \int_{\left(\frac{\bb{k}}{m}, \frac{\bb{k} + 1}{m}\right]} \hspace{-0.5mm} f(\bb{y}) \rd \bb{y}\right).
    \end{equation}
    By a union bound on $\bb{k}\in \N_0^d \cap (m - 1)\mathcal{S}$ (there are at most $m^d$ such points), and Bernstein's inequality (see, e.g., Lemma~\ref{lem:Bernstein.inequality}), we have, for all $\rho > 0$,
    \begin{equation}\label{eq:thm:Theorem.3.1.Babu.Canty.Chaubey.end}
        \PP\left(|L_{n,m}| > \rho m^{-1/2} \alpha_n\right)
        \leq m^d \cdot 2 \exp\left(-\frac{n^2 \rho^2 m^{-1} \alpha_n^2 / 2}{n \cdot c \cdot m^{-1} + \frac{1}{3} \cdot 1 \cdot n \rho m^{-1/2} \alpha_n}\right)
        \leq m^d \cdot 2 n^{-\rho^2 / (4c)},
    \end{equation}
    where the second inequality assumes that $m \leq n / \log n$ (equivalently, $\alpha_n \leq m^{-1/2}$), and $c\geq \rho$ is a Lipschitz constant for~$f$.
    If we choose $\rho = \rho(c,d) > 0$ large enough, then the right-hand side of \eqref{eq:thm:Theorem.3.1.Babu.Canty.Chaubey.end} is summable in $n$ and the Borel-Cantelli lemma implies $\|\hat{f}_{n,m} - f_m\|_{\infty} \leq m^d \, |L_{n,m}| = \OO(m^{d - 1/2} \alpha_n)$ a.s., as $n\to\infty$.
\end{proof}

\begin{proof}[\bf Proof of Theorem~\ref{thm:Theorem.3.2.and.3.3.Babu.Canty.Chaubey}]
    By \eqref{eq:equality.T.m.n}, the asymptotic normality of $n^{1/2} m^{-d/4} (\hat{f}_{n,m}(\bb{x}) - f_m(\bb{x}))$ will follow if we verify the Lindeberg condition for double arrays (see, e.g., Section~1.9.3. in \cite{MR0595165}):
    For every $\e > 0$,
    \begin{equation}\label{eq:prop:Proposition.1.Babu.Canty.Chaubey.Lindeberg.condition}
        s_m^{-2} \, \EE\big[|Y_{1,m}|^2 \ind_{\{|Y_{1,m}| > \e n^{1/2} s_m\}}\big] \longrightarrow 0, \quad n\to \infty,
    \end{equation}
    where $s_m^2 \leqdef \EE\big[|Y_{1,m}|^2\big]$ and $m = m(n) \to \infty$.
    Clearly, from \eqref{eq:def.Y.i.m},
    \begin{equation}
        |Y_{1,m}| \leq \max_{\bb{k}\in \N_0^d \cap (m-1)\mathcal{S}} 2 \, P_{\bb{k},m}(\bb{x}) = \OO(m^{-d/2}),
    \end{equation}
    and we also know that $s_m = m^{-3d/4} \sqrt{\psi(\bb{x}) f(\bb{x})} \, (1 + \oo_{\bb{x}}(1))$ when $f$ is Lipschitz continuous, by the proof of Theorem~\ref{thm:bias.var.density}.
    Therefore, we have
    \begin{equation}\label{eq:prop:Proposition.1.Babu.Canty.Chaubey.Lindeberg.condition.verify}
        \frac{|Y_{i,m}|}{n^{1/2} s_m} = \OO_{\bb{x}}(n^{-1/2} m^{3d/4} m^{-d/2}) = \OO_{\bb{x}}(n^{-1/2} m^{d/4}) \longrightarrow 0,
    \end{equation}
    whenever $n^{1/2} m^{-d/4}\to \infty$ as $m,n\to \infty$. (The bound on $|Y_{1,m}|$ in the proof of Proposition~1 in \cite{MR1910059} is suboptimal for $d = 1$. This is why we get a slightly better rate in \eqref{eq:prop:Proposition.1.Babu.Canty.Chaubey.Lindeberg.condition.verify} compared to the fourth equation on page 386 of \cite{MR1910059}.)
    Under this condition, \eqref{eq:prop:Proposition.1.Babu.Canty.Chaubey.Lindeberg.condition} holds (since for any given $\e > 0$, the indicator function is eventually equal to $0$ uniformly in $\omega\in \Omega$) and thus
    \begin{equation}
        n^{1/2} m^{-d/4} (\hat{f}_{n,m}(\bb{x}) - f_m(\bb{x}))
        = n^{1/2} m^{3d/4} (1 + \OO(m^{-1})) \cdot \frac{1}{n} \sum_{i=1}^n Y_{i,m}
        \stackrel{\mathscr{D}}{\longrightarrow} \mathcal{N}(0,f(\bb{x})\psi(\bb{x})).
    \end{equation}
    This completes the proof of Theorem~\ref{thm:Theorem.3.2.and.3.3.Babu.Canty.Chaubey}.
\end{proof}

\section{Technical lemmas and tools}\label{sec:tools}

The first lemma is a standard (but very useful) concentration bound, found for example in \cite[Corollary~2.11]{MR3185193}.

\begin{lemma}[Bernstein's inequality]\label{lem:Bernstein.inequality}
    Let $X_1, \dots, X_n$ be a sequence of independent random variables satisfying $|X_i| \leq b < \infty$.
    Then, for all $t > 0$,
    \begin{equation}
        \PP\left(\left|\sum_{i=1}^n (X_i - \EE[X_i])\right| \geq t\right) \leq 2 \, \exp\left(-\frac{t^2 / 2}{\sum_{i=1}^n \VV(X_i) + \frac{1}{3} b t}\right).
    \end{equation}
\end{lemma}

In the second lemma, we estimate sums of powers of multinomial probabilities.
This is used in the proof of Theorem~\ref{thm:bias.var.density} and the proof of Theorem~\ref{thm:MISE.optimal.density}.

\begin{lemma}\label{lem:technical.sums}
    For every $\bb{x}\in \mathrm{Int}(\mathcal{S})$, we have, as $r\to \infty$,
    \begin{align}
        &r^{\hspace{0.3mm}d/2} \sum_{\bb{k}\in \N_0^d \cap r\mathcal{S}} P_{\bb{k},r}^{\hspace{0.2mm}2}(\bb{x}) = \left[(4\pi)^d (1 - \|\bb{x}\|_1) \prod_{i=1}^d x_i\right]^{-1/2} + \oo_{\bb{x}}(1), \label{eq:lem:technical.sums.eq.1} \\
        &r^{\hspace{0.3mm}d} \sum_{\bb{k}\in \N_0^d \cap r\mathcal{S}} P_{\bb{k},r}^{\hspace{0.2mm}3}(\bb{x}) = \left[(2\sqrt{3} \hspace{0.2mm}\pi)^d (1 - \|\bb{x}\|_1) \prod_{i=1}^d x_i\right]^{-1} + \oo_{\bb{x}}(1). \label{eq:lem:technical.sums.eq.2}
    \end{align}
\end{lemma}

\begin{proof}[\bf Proof of Lemma~\ref{lem:technical.sums}]
    It is well known that the covariance matrix of the multinomial distribution is $r \, \Sigma_{\bb{x}}$, where $\Sigma_{\bb{x}} = \text{diag}(\bb{x}) - \bb{x} \bb{x}^{\top}$,
    see, e.g., \cite[p.377]{MR2168237}, and it is also known that
    \begin{equation}
        \det(\Sigma_{\bb{x}}) = (1 - \|\bb{x}\|_1) \prod_{i=1}^d x_i,
    \end{equation}
    see, e.g., \cite[Theorem~1]{MR1157720}.
    Therefore, consider
    \begin{equation}\label{eq:phi.M}
        \phi_{\Sigma_{\bb{x}}}(\bb{y}) \leqdef \frac{1}{\sqrt{(2\pi)^{\hspace{0.2mm}d} \det(\Sigma_{\bb{x}})}} \, \exp\left(-\frac{1}{2} \bb{y}^{\top} \Sigma_{\bb{x}}^{-1} \, \bb{y}\right), \quad \bb{y}\in \R^d,
    \end{equation}
    the density of the multivariate normal $\mathcal{N}(\bb{0},\Sigma_{\bb{x}})$.
    By a local limit theorem for the multinomial distribution (see, e.g., Lemma~2 in \cite{MR0478288} or Theorem~2.1 in \cite{MR4249129}), we have
    \begin{align}
        r^{\hspace{0.3mm}d/2} \sum_{\bb{k}\in \N_0^d \cap r\mathcal{S}} P_{\bb{k},r}^{\hspace{0.2mm}2}(\bb{x})
        = \int_{\R^d} \phi_{\Sigma_{\bb{x}}}^2(\bb{y}) \rd \bb{y} + \oo_{\bb{x}}(1)
        = \frac{2^{-d/2}}{\sqrt{(2\pi)^{\hspace{0.2mm}d} \det(\Sigma_{\bb{x}})}} \int_{\R^d} \phi_{\frac{1}{2} \Sigma_{\bb{x}}}(\bb{y}) \rd \bb{y} + \oo_{\bb{x}}(1)
        = \frac{2^{-d/2}}{\sqrt{(2\pi)^{\hspace{0.2mm}d} \det(\Sigma_{\bb{x}})}} \cdot 1 + \oo_{\bb{x}}(1),
    \end{align}
    and
    \begin{align}
        r^{\hspace{0.3mm}d} \sum_{\bb{k}\in \N_0^d \cap r\mathcal{S}} P_{\bb{k},r}^{\hspace{0.2mm}3}(\bb{x})
        = \int_{\R^d} \phi_{\Sigma_{\bb{x}}}^3(\bb{y}) \rd \bb{y} + \oo_{\bb{x}}(1)
        = \frac{3^{-d/2}}{(2\pi)^{\hspace{0.2mm}d} \det(\Sigma_{\bb{x}})} \int_{\R^d} \phi_{\frac{1}{3} \Sigma_{\bb{x}}}(\bb{y}) \rd \bb{y} + \oo_{\bb{x}}(1)
        = \frac{3^{-d/2}}{(2\pi)^{\hspace{0.2mm}d} \det(\Sigma_{\bb{x}})} \cdot 1 + \oo_{\bb{x}}(1).
    \end{align}
    This ends the proof.
\end{proof}

In the third lemma, we estimate another technical sum, needed in the proof Theorem~\ref{thm:bias.var} and the proof of Theorem~\ref{thm:MISE.optimal}.

\begin{lemma}\label{lem:technical.sums.R}
    For $i\in \{1,\dots,d\}$ and $r\in \N$, let
    \begin{equation}\label{eq:lem:technical.sums.R}
        R_{i,r}(\bb{x}) \leqdef r^{1/2} \hspace{-1mm} \sum_{\bb{k},\bb{\ell}\in \N_0^d \cap r\mathcal{S}} \Big(\frac{k_i \wedge \ell_i}{r} - x_i\Big) P_{\bb{k},r}(\bb{x}) P_{\bb{\ell},r}(\bb{x}), \quad \bb{x}\in \mathcal{S}.
    \end{equation}
    Then,
    \begin{equation}\label{eq:lem:technical.sums.R.claim.1}
        \sup_{1 \leq i \leq d} \sup_{r\in \N} \, \sup_{\bb{x}\in \mathcal{S}} \, |R_{i,r}(\bb{x})| \leq 1,
    \end{equation}
    and for every $\bb{x}\in \mathrm{Int}(\mathcal{S})$, we have,
    \begin{equation}\label{eq:lem:technical.sums.R.claim.2}
        R_{i,r}(\bb{x}) = -\sqrt{\frac{x_i (1 - x_i)}{\pi}} + \oo_{\bb{x}}(1), \quad r\to \infty.
    \end{equation}
\end{lemma}

\begin{proof}[\bf Proof of Lemma~\ref{lem:technical.sums.R}]
    By the Cauchy-Schwarz inequality and the identity \eqref{eq:var.multinomial.identity}, we have
    \begin{align}
        |R_{i,r}(\bb{x})|
        \leq 2 r^{1/2} \hspace{-1mm} \sum_{\bb{k}\in \N_0^d \cap r\mathcal{S}} \Big|\frac{k_i}{r} - x_i\Big| P_{\bb{k},r}(\bb{x})
        \leq 2 r^{1/2} \hspace{-1mm} \sqrt{\sum_{\bb{k}\in \N_0^d \cap r\mathcal{S}} \Big|\frac{k_i}{r} - x_i\Big|^2 P_{\bb{k},r}(\bb{x})}
        \leq 2 r^{1/2} \sqrt{\frac{x_i (1 - x_i)}{r}} \leq 1.
    \end{align}
    For the second claim, we know that the marginal distributions of the multinomial are binomial, so if $\phi_{\sigma^2}$ denotes the density function of the $\mathcal{N}(0,\sigma^2)$ distribution, a standard local limit theorem for the binomial distribution (see, e.g., \citet{MR56861} or Theorem~2.1 in \cite{MR4249129}) and integration by parts show that
    \begin{align}\label{eq:thm:bias.var.end.2}
        R_{i,r}(\bb{x})
        &= 2 \cdot x_i (1 - x_i) \int_{-\infty}^{\infty} \frac{z}{x_i (1 - x_i)} \, \phi_{x_i(1 - x_i)}(z) \int_z^{\infty} \phi_{x_i(1 - x_i)}(y) \rd y \rd z + \oo_{\bb{x}}(1) \notag \\[1.5mm]
        &= 2 \cdot x_i (1 - x_i) \, \left[0 - \int_{-\infty}^{\infty} \phi_{x_i (1 - x_i)}^2(z) \rd z\right] + \oo_{\bb{x}}(1)
        = \frac{- 2 x_i (1 - x_i)}{\sqrt{2\pi \cdot 2 x_i(1 - x_i)}} \int_{-\infty}^{\infty} \phi_{\frac{1}{2} x_i (1 - x_i)}(z) \rd z + \oo_{\bb{x}}(1) \\
        &= -\sqrt{\frac{x_i (1 - x_i)}{\pi}} + \oo_{\bb{x}}(1).
    \end{align}
    This ends the proof.
\end{proof}

\begin{remark}\label{eq:correction.Leblanc.Belalia}
    The proof of \eqref{eq:lem:technical.sums.R.claim.2} is much simpler here than the proof of Lemma~2\hspace{0.5mm}(iv) in \cite{MR2960952} ($d = 1$), where a finely tuned continuity correction from \citet{MR538319} was used to estimate the survival function of the binomial distribution instead of working with a local limit theorem directly.
    There is also an error in Leblanc's paper (for an explanation, see Remark~3.4 in \cite{MR4213687}).
    His function $\psi_2(x)$ should be equal to
    \begin{equation}
        [x (1 - x) / (4\pi)]^{1/2} \quad \text{instead of} \quad [x (1 - x) / (2\pi)]^{1/2}.
    \end{equation}
    As a consequence, his function $V(x)$ should be equal to
    \begin{equation}
        f(x) \, [x (1 - x)/\pi]^{1/2} \quad \text{instead of} \quad f(x) \, [2x (1 - x)/\pi]^{1/2}.
    \end{equation}
    As pointed out in Remark~\ref{rem:error.Leblanc}, this error has spread to at least 15 papers/theses who relied on Lemma~2\hspace{0.3mm}(iv) in \cite{MR2960952}; the list appears with suggested corrections in Appendix~B of \cite{MR4213687}.
\end{remark}

In the fourth lemma, we prove the integral version of \eqref{eq:lem:technical.sums.eq.1}.
This is needed in the proof of Theorem~\ref{thm:MISE.optimal.density}.

\begin{lemma}[\citet{MR3825458}]\label{lem:lemma.4.Leblanc.2006.tech.report}
    We have, as $r\to \infty$,
    \begin{equation}
        r^{\hspace{0.3mm}d/2} \int_{\mathcal{S}} \sum_{\bb{k}\in \N_0^d \cap r\mathcal{S}} P_{\bb{k},r}^{\hspace{0.2mm}2}(\bb{x}) \rd \bb{x} = \frac{2^{-d} \sqrt{\pi}}{\Gamma(d/2 + 1/2)} + O(r^{-1}) = \int_{\mathcal{S}} \psi(\bb{x}) \rd \bb{x} + \OO(r^{-1}),
    \end{equation}
    where recall $\psi(\bb{x}) \leqdef \big[(4\pi)^d (1 - \|\bb{x}\|_1) \prod_{i=1}^d x_i\big]^{-1/2}$.
\end{lemma}

\begin{proof}[\bf Proof of Lemma~\ref{lem:lemma.4.Leblanc.2006.tech.report}]
    Throughout the proof, let $k_{d+1} \leqdef r - \|\bb{k}\|_1$.
    We have
    \begin{align}\label{eq:prop:lemma.4.Leblanc.2006.tech.report.beginning}
        \sum_{\bb{k}\in \N_0^d \cap r \mathcal{S}} \int_{\mathcal{S}} (P_{\bb{k},r}(\bb{x}))^2 \rd \bb{x}
        &= \sum_{\bb{k}\in \N_0^d \cap r \mathcal{S}} \left(\frac{\Gamma(r + 1)}{\prod_{i=1}^{d+1} \Gamma(k_i + 1)}\right)^2 \int_{\mathcal{S}} \, \prod_{i=1}^{d+1} x_i^{2k_i} \rd \bb{x}
        = \sum_{\bb{k}\in \N_0^d \cap r \mathcal{S}} \left(\frac{\Gamma(r + 1)}{\prod_{i=1}^{d+1} \Gamma(k_i + 1)}\right)^2 \frac{\prod_{i=1}^{d+1} \Gamma(2k_i + 1)}{\Gamma(2r + d + 1)} \notag \\
        &= \frac{(\Gamma(r + 1))^2}{\Gamma(2r + d + 1)} \sum_{\bb{k}\in \N_0^d \cap r \mathcal{S}} \prod_{i=1}^{d+1} \binom{2 k_i}{k_i}.
    \end{align}
    To obtain the third equality, we used the normalization constant for the Dirichlet distribution.
    Now, note that
    \begin{equation}\label{eq:prop:lemma.4.Leblanc.2006.tech.report.Graham.formula}
        \sum_{\bb{k}\in \N_0^d \cap r \mathcal{S}} \prod_{i=1}^{d+1} \binom{2 k_i}{k_i}
        = (-4)^r \sum_{\bb{k}\in \N_0^d \cap r \mathcal{S}} \prod_{i=1}^{d+1} \frac{1}{(-4)^{k_i}} \binom{2 k_i}{k_i}
        = (-4)^r \sum_{\bb{k}\in \N_0^d \cap r \mathcal{S}} \prod_{i=1}^{d+1} \binom{-1/2}{k_i} = (-4)^r \, \binom{-(d+1)/2}{r} = \binom{r + \frac{d-1}{2}}{r} \, 4^r,
    \end{equation}
    where the last three equalities follow, respectively, from (5.37), the Chu-Vandermonde convolution formula (p.\hspace{-1mm} 248), and (5.14) in \cite{MR1397498}.
    By applying \eqref{eq:prop:lemma.4.Leblanc.2006.tech.report.Graham.formula} and the duplication formula
    \begin{equation}\label{eq:prop:lemma.4.Leblanc.2006.tech.report.duplication.formula}
        4^{-y} = \frac{\Gamma(y) \Gamma(y + 1/2)}{2 \sqrt{\pi} \, \Gamma(2y)}, \quad y\in (0,\infty),
    \end{equation}
    see \cite[p.256]{MR0167642}, in \eqref{eq:prop:lemma.4.Leblanc.2006.tech.report.beginning}, we get
    \begin{align*}
        \int_{\mathcal{S}} \sum_{\bb{k}\in \N_0^d \cap r \mathcal{S}} (P_{\bb{k},r}(\bb{x}))^2 \rd \bb{x}
        &= \frac{(\Gamma(r + 1))^2}{\Gamma(2r + d + 1)} \cdot \frac{\Gamma(r + d/2 + 1/2)}{\Gamma(r + 1)\Gamma(d/2 + 1/2)} \cdot 4^r \\[-0.5mm]
        &= \frac{2 \sqrt{\pi} \, \Gamma(r + 1)}{\Gamma(d/2 + 1/2) \Gamma(r + d/2 + 1)} \cdot \frac{\Gamma(r + d/2 + 1/2) \Gamma(r + d/2 + 1)}{2 \sqrt{\pi} \, \Gamma(2r + d + 1)} \cdot 4^r \\[2mm]
        &= \frac{2 \sqrt{\pi} \, \Gamma(r + 1)}{\Gamma(d/2 + 1/2) \Gamma(r + d/2 + 1)} \cdot \frac{4^r}{4^{r + d/2 + 1/2}}
        = \frac{2^{-d} \sqrt{\pi} \, \Gamma(r + 1)}{\Gamma(d/2 + 1/2) \Gamma(r + d/2 + 1)} \\[2mm]
        &= \left\{\hspace{-1mm}
        \begin{array}{ll}
            \frac{2^{-d} \sqrt{\pi}}{\Gamma(d/2 + 1/2)} \prod_{i=1}^{d/2} (r + i)^{-1}, &\mbox{if } d ~\text{is even}, \\[2mm]
            \frac{2^{-d} \sqrt{\pi}}{\Gamma(d/2 + 1/2)} \prod_{i=1}^{d/2+1/2} (r + d/2 + 1 - i)^{-1} \cdot \frac{\Gamma(r + 1)}{\Gamma(r + 1/2)}, &\mbox{if } d ~\text{is odd}. \\
        \end{array}
        \right.
    \end{align*}
    Using the fact that
    \begin{equation}
        \frac{\Gamma(r + 1)}{r^{1/2} \Gamma(r + 1/2)} = 1 + \frac{1}{8r} + O(r^{-2}),
    \end{equation}
    see \cite[p.257]{MR0167642}, we obtain
    \begin{equation}\label{prop:lemma.4.Leblanc.2006.tech.report.S.m.asymptotic}
        r^{\hspace{0.3mm}d/2} \int_{\mathcal{S}} \sum_{\bb{k}\in \N_0^d \cap r \mathcal{S}} (P_{\bb{k},r}(\bb{x}))^2 \rd \bb{x} = \frac{2^{-d} \sqrt{\pi}}{\Gamma(d/2 + 1/2)} + O(r^{-1}).
    \end{equation}
    On the other hand,
    \begin{align}
        \int_{\mathcal{S}} \Big[(4\pi)^d (1 - \|\bb{x}\|_1) \prod_{i=1}^d x_i\Big]^{-1/2} \rd \bb{x}
        &= \frac{1}{2^d \pi^{d/2}} \int_{\mathcal{S}} \prod_{i=1}^{d+1} x_i^{1/2 - 1} \rd \bb{x}
        = \frac{1}{2^d \pi^{d/2}} \cdot \frac{(\Gamma(1/2))^{d+1}}{\Gamma(d/2 + 1/2)}
        = \frac{2^{-d} \sqrt{\pi}}{\Gamma(d/2 + 1/2)}.
    \end{align}
    Together with \eqref{prop:lemma.4.Leblanc.2006.tech.report.S.m.asymptotic}, this ends the proof.
\end{proof}

\section*{Acknowledgments}

The author is supported by a postdoctoral fellowship from the NSERC (PDF) and the FRQNT (B3X supplement).
We thank the Editor, Associate Editor and referees, as well as our financial sponsors.

%
%


\bibliographystyle{myjmva}
\bibliography{Ouimet_2020_review_Bernstein_and_asymmetric_kernels_bib}

\end{document}